\documentclass[11pt]{amsart}
\usepackage{bm}
\usepackage{graphicx}
\usepackage{mathpazo}
\usepackage{latexsym}   
\usepackage{amssymb}    
\usepackage{amsmath}    
\usepackage{amsthm}     
\usepackage{hyperref}
\usepackage[margin=1in]{geometry}
\usepackage{color}
\usepackage{xcolor}
\usepackage[cmtip,all,matrix,arrow,tips,curve]{xy}
\usepackage[final,notcite, notref]{showkeys}
\usepackage{filecontents}

\DeclareFontFamily{U}{wncy}{} 
\DeclareFontShape{U}{wncy}{m}{n}{<->wncyr10}{}
\DeclareSymbolFont{mcy}{U}{wncy}{m}{n}
\DeclareMathSymbol{\Sh}{\mathord}{mcy}{"58}

\renewcommand{\bar}[1]{{\overline{#1}}}

\newcommand{\orbit}[2]{^{#1}#2}

\DeclareMathOperator{\disc}{disc}
\DeclareMathOperator{\Norm}{N}
\DeclareMathOperator{\charpoly}{charpoly}
\def\tor{\text{tor}}
\def\et{\text{\'et}}

\def\satake{{\mathfrak S}}

\def\res{\bm{{\rm R}}}

\DeclareMathOperator{\diag}{diag}

\def\one{{\mathbb 1}}

\def\wnum{\widetilde{\#}}

\newcommand{\Q}{{\mathbb Q}}
\newcommand{\C}{{\mathbb C}}
\newcommand{\R}{{\mathbb R}}
\newcommand{\Z}{{\mathbb Z}}

\newcommand{\F}{{\mathbb F}}

\newcommand{\GL}{\mathrm {GL}}
\newcommand{\fg}{{\mathfrak g}}

\newcommand{\A}{{\mathbb A}}

\newcommand{\steinberg}{\mathrm{geom}}
\newcommand{\can}{\mathrm{can}}
\newcommand{\tama}{\mathrm{Tama}}

\newcommand{\geom}{\mathrm{geom}}
\newcommand{\serre}{\mathrm{SO}}
\newcommand{\spl}{\mathrm{spl}}

\newcommand\cX{{\mathcal X}}

\newcommand\Gal{\mathrm{Gal}}

\def\der{\mathrm{der}}

\def\idp{{\mathfrak p}}
\def\iso{{\cong}}
\def\res{\bm{{\rm R}}}
\def\resone{\res^{(1)}}

\def\G{{\mathbb G}}

\def\aff{{\mathbb A}}
\def\cx{{\mathbb C}}
\def\ff{{\mathbb F}}
\def\gp{{\mathbb G}}
\def\rat{{\mathbb Q}}

\def\integ{{\mathbb Z}}
\DeclareMathOperator{\mat}{Mat}
\DeclareMathOperator{\SP}{Sp}

\DeclareMathOperator{\gsp}{GSp}
\DeclareMathOperator{\ord}{ord}
\DeclareMathOperator{\gl}{GL}
\DeclareMathOperator{\SL}{SL}
\DeclareMathOperator{\aut}{Aut}
\DeclareMathOperator{\vol}{vol}
\DeclareMathOperator{\End}{End}
\DeclareMathOperator{\rank}{rank}
\DeclareMathOperator{\tr}{tr}
\newcommand{\til}[1]{{\widetilde{#1}}}
\newcommand{\invlim}[1]{\lim_{\stackrel{\leftarrow}{#1}}}
\newcommand{\oneover}[1]{\frac{1}{#1}}
\def\twiddle{\sim}
\def\cris{{\rm cris}}
\def\inv{^{-1}}
\def\ra{\to}
\newcommand{\half}[1]{\frac{#1}{2}}
\def\cross{\times}
\def\units{^\cross}
\def\tensor{\otimes}
\def\calh{{\mathcal H}}
\def\calo{{\mathcal O}}
\newcommand{\ubar}[1]{{\underline{#1}}}
\newcommand{\st}[1]{\{#1\}}
\newcommand{\abs}[1]{{\left|#1\right|}}
\newcommand{\ang}[1]{{{\langle #1 \rangle}}}
\newcommand{\rest}[1]{|_{#1}}

\newenvironment{alphabetize}{\begin{enumerate}

}{\end{enumerate}}

\def\level{\Gamma}

\def\transpose{^\top}

\newtheorem{theorem}{Theorem}[section]
\newtheorem{lemma}[theorem]{Lemma}
\newtheorem{proposition}[theorem]{Proposition}
\newtheorem{corollary}[theorem]{Corollary}

\newtheorem{defn}[theorem]{Definition}
\newtheorem{theoremintro}{Theorem}

\theoremstyle{remark}
\newtheorem{rem}[theorem]{Remark}

\numberwithin{equation}{section}

\newcommand{\stk}[1]{{\mathcal #1}}
\def\stack{\stk}
\newcommand{\fc}{{\mathfrak c}}

\begin{document}

\title{Counting abelian varieties over finite fields via Frobenius
  densities}

\author{Jeffrey D. Achter}
\address{Colorado State University, Fort Collins, CO, USA}
\email{achter@math.colostate.edu}

\author{S.~Ali Altu\u{g}}
\address{New York, NY, USA}
\email{ali.altug@gmail.com}

 \author{Luis Garcia}
\address{University College London, London, UK}
 \email{l.e.garcia@ucl.ac.uk}

\author{Julia Gordon}
\address{University of British Columbia, Vancouver, BC, Canada}
\email{gor@math.ubc.ca}
\begin{abstract}
Let $[X,\lambda]$ be a principally polarized abelian variety over a
finite field with commutative endomorphism ring; further suppose that
either $X$ is ordinary or the field is prime.  Motivated by an
equidistribution heuristic, we introduce a factor $\nu_v([X,\lambda])$
for each place $v$ of $\rat$, and show that the product of these
factors essentially computes the size of the isogeny class of
$[X,\lambda]$.

The derivation of this mass formula depends on a formula of Kottwitz
and on analysis of measures on the group of symplectic similitudes, and
in particular does not rely on a calculation of class numbers.

\end{abstract}

\maketitle

\section{Introduction}\label{secintro}

Let $[X,\lambda] \in \stk A_g(\ff_q)$ be a principally polarized
$g$-dimensional abelian variety over the finite field
$\ff_q = \ff_{p^e}$.  Its isogeny class $I([X,\lambda],\ff_q)$ is finite; our goal is to understand the
(weighted by automorphism group) cardinality
$\wnum I([X,\lambda],\ff_q)$.

A random matrix heuristic might suggest the following.  Let
$f_{X/\ff_q}(T)$ be the characteristic polynomial of Frobenius of
$X$.  It is well-known that $f_{X/\ff_q}(T)\in \Z[T]$. Following Gekeler \cite{gekeler03}, for a rational prime $\ell \nmid p \disc(f)$, one can define a number
\begin{equation}
\label{eqdefnuintronaive}
\nu_\ell([X,\lambda],\ff_q) =\lim_{n\rightarrow \infty} 
\frac{\#\st{ \gamma \in \gsp_{2g}(\integ_\ell/\ell^n) :
    \charpoly_\gamma(T) = f_{X/\ff_q}(T) \bmod \ell^n}}{\#
  \gsp_{2g}(\integ_\ell/\ell^n)/\#\mathbb A_{\gsp_{2g}}(\integ/\ell^n)},
\end{equation}
where $\gsp_{2g}$ is the group of symplectic similitudes of a
symplectic space of dimension $2g$, and $\mathbb A_{\gsp_{2g}}$ is the
space of characteristic polynomials of these similitudes.

For $\ell\nmid p\disc(f)$, in which case the conjugacy class is determined by the characteristic polynomial (cf. Lemma \ref{lem:gqvsgz}), we interpret $\nu_{\ell}[X,\lambda]$ as the deviation of the size of the conjugacy class with characteristic polynomial $f_{X/\ff_q}(T)$ from the average size of a conjugacy class in $\gsp(\Z_{\ell})$.

For $\ell \mid
\disc(f_{X/\ff_q}(T))$, since the characteristic polynomial need not
determine a unique conjugacy class in $\gsp_{2g}(\integ_\ell)$, a slightly more
involved definition of $\nu_{\ell}[X,\lambda]$ is needed, see \eqref{eq:defgekeler}.
Similarly, we define
quantities $\nu_p([X,\lambda],\ff_q)$ and $\nu_\infty([X,\lambda],\ff_q)$ using, respectively,
equidistribution considerations for $\sigma$-conjugacy classes in
$\gsp_{2g}(\rat_q)$ and the Sato-Tate measure on the compact form
$\operatorname{USp}_{2g}$.  

Careless optimism might lead one to hope
that $\wnum I([X,\lambda],\ff_q)$ is given by the product of the average archimedean and $p$-adic masses with the local deviations:
\begin{equation}
\label{eqmainnotau}
\wnum I([X,\lambda],\ff_q) \propto \nu_\infty([X,\lambda],\ff_q) \prod_\ell
\nu_\ell([X,\lambda],\ff_q).
\end{equation}
This argument is (at best) superficially plausible. Nonetheless, in this paper we give a pure-thought proof of the following theorem:
\begin{theoremintro}
\label{thmain}
Let $[X,\lambda]$ be a principally polarized abelian variety over
$\ff_q$ with commutative endomorphism ring.   Suppose that either $X$
is ordinary or that $\ff_q = \ff_p$ is the prime field.
Then 
\begin{equation}
\label{eqmain}
\wnum I([X,\lambda],\ff_q) = q^{\frac{\dim(\mathcal{A}_g)}{2}}\tau_T \nu_\infty([X,\lambda],\ff_q) \prod_\ell
\nu_\ell([X,\lambda],\ff_q).
\end{equation}
Here $\dim(\mathcal{A}_g)=\frac{g(g+1)}{2}$ and $\tau_T$ is the
Tamagawa number of the algebraic torus associated with $[X, \lambda]$
in \S \ref{subsec:kottwitz}.
\end{theoremintro}

As we have mentioned, this formulation is inspired by \cite{gekeler03}, in which Gekeler
proves Theorem \ref{thmain} for an ordinary elliptic curve $E$ over a finite
prime field $\ff_p$.  
(In the case $g=1$ considered by Gekeler, $\tau_T$ equals $1$.)
Roughly speaking, the strategy there is to
compute the terms $\nu_\ell$ explicitly, and show that the right-hand
side of \eqref{eqmain} actually computes, via Euler products, the
value at $s=1$ of a suitable L-function.  One concludes via the
analytic class number formula and the known description of the isogeny 
class $I(E,\ff_q)$ as a torsor under the class group of the quadratic imaginary order attached to the Frobenius of $E$.  This
strategy was redeployed in \cite{achterwilliams15} and \cite{gerhardwilliams19} for certain ordinary abelian varieties.

More recently, in \cite{achtergordon17}, the first- and last-named
authors
showed directly that the right-hand side of
\eqref{eqmain} actually computes the product of the volume of a certain (adelic) quotient and an orbital integral on
$\gl_2$.  Thanks to the work of Langlands \cite{langlands:antwerp}, and Dirichlet's class number formula, one has a direct proof that
this product computes the size of the isogeny class of the
elliptic curve.

In fact, this formula of Langlands, originally developed to count
points on modular curves over finite fields, has been generalized by 
Kottwitz to an essentially arbitrary Shimura variety of PEL type
\cite{kottwitz92}.  Kottwitz's formula (see
Proposition \ref{proplk} below), as in the case of Langlands, comes as a product of an (adelic) volume of a torus and an orbital integral, this time over $\gsp_{2g}$. Let us remark that although the orbital integral in Kottwitz's and Langlands' formulas clearly decomposes as a product of local terms, the volume term, however, appears as a global quantity (a class number in the case of $\gl_2$, cf. Lemma A.4 of \cite{achtergordon17}). Therefore an Euler product expression for $\wnum I([X,\lambda],\ff_q)$ such as the one in \eqref{eqmain} is, at least, not immediate. 

The content of the present paper is to prove that the Euler product given by the right-hand side of \eqref{eqmain} is indeed equal to the product of the global volume and the orbital integral given by Kottwitz's formula. We establish this by 
a delicate analysis of the interplay between various measures on the relevant spaces.

This paper is the logical extension of \cite{achtergordon17}, which
worked out these details for the case where the governing group is
$\gsp_2 = \gl_2$.  The reader will correctly expect that the structure
of the argument is largely similar.  However, the cohomological and
combinatorial intricacies of symplectic similitude groups in
comparison to general linear groups -- in particular, the tori are much more complex and 
conjugacy and stable conjugacy need not coincide -- mean that each
stage is considerably more involved.  

We highlight three particular issues that make the generalization from elliptic curves to higher rank not straight-forward. 

The first is already mentioned above -- the difference between conjugacy and stable conjugacy in $\gsp_{2g}$ when $g>1$.
This issue is discussed in detail in Section \ref{sec:conj}, and leads to the definition \ref{def:gekratio}, which (as we prove in Section \ref{sec:conj}) coincides with \eqref{eqdefnuintronaive} when $\ell \nmid p \disc(f)$.

The second is the fundamental lemma for base change, which is used to relate a Gekeler-style ratio at $p$ to the twisted orbital integral. The complicated function one generally gets as a result of base change  is the reason we have to assume that $X$ is ordinary if $q\neq p$; this is discussed in detail in \ref{sub:lisp}.

The last is that the tori in $\gsp_{2g}$ for $g\geq2$ are significantly more complicated than those for $g=1$. 
The global calculation in Section \ref{sec:global} reflects this complexity, and involves the Tamagawa number of the algebraic torus $T$.
This number is well-known to be $1$ for $g=1$, but for general $g$
we have to leave it as an (unknown) constant; Thomas R\"ud and (independently) Wen-Wei Li  obtained suggestive partial results and kindly agreed to present them in Appendix \ref{sec:tamagawa}.

Perhaps not surprisingly, \eqref{eqmain}
can also be interpreted as a Smith-Minkowski-Siegel type mass formula
(in the sense of Tamagawa-Weil) with explicit local masses
(cf. \cite{gan-yu}). Here the underlying group, of course, is
$\gsp_{2g}$ and the masses calculate sizes of the relevant isogeny
classes. Although this point of view is interesting in its own right
we do not pursue it further in this paper. We would,
however, like to note that the appearance of Tamagawa numbers is
natural in this context.

The present work is, of course, part of a long and thriving discourse
on the size of isogeny classes of abelian varieties over finite
fields.  Deuring computes the size of an isogeny class of elliptic
curves as a class number \cite{deuring41} .  Waterhouse reinterprets
and extends Deuring's work to a much larger class of abelian
varieties \cite{waterhouse:69}.  The key point in his study is to consider separately the
$\ell$-primary components of the kernel of an isogeny, and model them
using the various Tate modules (for $\ell$ prime to the
characteristic) and the Dieudonn\'e module.  Both of these approaches
have been revisited and enhanced in recent times.  Yu
(and collaborators) have undertaken a detailed analysis of, for example, the number of
supersingular abelian varieties over finite fields; their answers are
often expressed in terms of a mass formula which comes from a
local lattice-theoretic perspective on the isogeny problem (e.g., \cite{xueyu21,yu12}).  In a somewhat
different direction, Marseglia \cite{marseglia21} and Howe \cite{howevariation}
are often able to express the size of an
isogeny class of principally polarized abelian varieties as a sum of
suitably generalized class numbers. 

In the approach taken
here, the orbital integral in the Langlands--Kottwitz formula
(Proposition \ref{proplk}) does the work of assimilating information
from the different local lattice calculations. 
Class numbers
necessarily arise from our formula, but are not built-in;  see
\S \ref{S:examples} for this emergence.

\subsection*{Acknowledgment} 
JDA would like to thank Edgar Costa, Stefano Marseglia, David Roe and
Christelle Vincent for sharing preliminary data on isogeny classes.
SAA would like to thank Sungmun Cho and 
Jacob Tsimerman for various discussions. JG is grateful to William
Casselman, Sug-Woo Shin, and Nicolas Templier for several helpful
discussions, and to Cornell University,  Association for Women in
Mathematics, and the Michler family, as part of her work was carried
out during a semester at Cornell University enabled by the Ruth
Michler memorial prize.   Our collaboration also benefited from NSERC
funding, and our paper benefited from the referee's careful reading.
SAA's work was partially supported by the NSF (DMS1702176) and a
startup grant from Boston University, and JDA's work was partially supported by grants from the NSA (H98230-14-1-0161,
H98230-15-1-0247 and H98230-16-1-0046).

\subsection*{Notation}

We work over a finite field $\ff_q = \ff_{p^e}$ of characteristic
$p$. We will often drop the field from all notation.  We denote by
$\rat_q$ the degree $e$ unramified extension of $\rat_p$, and by
$\integ_p$ its ring of integers; the $p^{th}$ power automorphism of
$\ff_q$ lifts to the Frobenius automorphism $\sigma$ of $\rat_q$.

Fix a positive integer $g$.  Let $V = \integ^{\oplus 2g}$, endowed
with basis $x_1, \cdots, x_g, y_1, \cdots y_g$, and equip it with the
symplectic form such that $\ang{x_i,y_j} = \delta_{ij}$ and
$\ang{x_i,x_j} = \ang{y_i,y_j} =0$.  Let $G$ be the group of
similitudes $G = \gsp(V,\ang{\cdot,\cdot}) \iso \gsp_{2g}$; it has
dimension $2g^2+g+1$ and rank $r = g+1$.  Its derived group is
$G^\der \iso \SP_{2g}$, and $G/G^\der \iso \gp_m$.  Let $\eta :G \to
\gp_m$ be the corresponding surjection (the multiplier map). We write $T_{\spl}$ for the split torus of diagonal matrices in $G$.

If $K$ is a field, $\alpha \in G(K)$, and $\Gamma \subseteq G(K)$, we
let $\orbit\Gamma\alpha = \st{ \beta\inv \alpha \beta: \beta \in
  \Gamma}$ be the orbit of $\alpha$ under $\Gamma$. Since $G$ has
simply connected derived group, the stable conjugacy class, or stable
orbit, of $\alpha$ is those
elements of $G(K)$ which are conjugate to $\alpha$ as elements of
$G(\bar K)$ (\cite[p.785]{kottwitz82}).

For $\alpha \in G(\rat_q)$, the twisted, or $\sigma$-, conjugacy class
of $\alpha$ is $\st{ \beta\inv \alpha \beta^\sigma: \beta \in
  G(\rat_q)}$; and the stable twisted conjugacy class of $\alpha$ is
$G(\rat_q)\cap \st{ \beta\inv \alpha\beta^\sigma: \beta \in G(\bar \rat_q)}$.

For an element $\gamma\in G(\Q_\ell)$, where $\ell$ is an arbitrary prime,
the Weyl discriminant of $\gamma$ is denoted 
by $D(\gamma)$: 
$D(\gamma)= \prod_{\alpha\in \Phi}(1-\alpha(\gamma))$, where the product is over all roots of $G$ (see \S \ref{subsub:weyldisc} for details).

\section{Background}\label{secdef}

\subsection{The Kottwitz formula}
\label{subsec:kottwitz}

The key formula we need is developed by Kottwitz in
\cite{kottwitz92}.  In fact, the special case we need is detailed  in
\cite[Sec. 12]{kottwitz90}.  By way of establishing necessary
notation, we review the relevant part of this work here.

Let $\stk A_g$ denote the moduli space of principally polarized
abelian varieties of dimension $g$.  An isogeny between two
principally polarized abelian varieties $[X,\lambda],[Y,\mu] \in \stk
A_g(\ff_q)$ is an isogeny $\phi:X \to Y$ such that $m\phi^*\mu = n \cdot
\lambda$ for some nonzero integers $m$ and $n$.  The isogeny class
$I([X,\lambda],\ff_q)$ is the set of all principally
polarized abelian varieties $[Y,\mu]/\ff_q$ admitting such an isogeny
(over $\ff_q$), and its weighted cardinality is
\[
\wnum I([X,\lambda],\ff_q) = \sum_{[Y,\mu] \in I([X,\lambda],\ff_q)}
\oneover{\#\aut(Y,\mu)}.
\]

The abelian variety $X/\ff_q$ admits a Frobenius endomorphism
$\varpi_{X/\ff_q}$, with characteristic polynomial $f_{X/\ff_q}(T)$ of
degree $2g$.  (By \cite{tateendff}, this polynomial determines the isogeny class of $X$ as
an unpolarized abelian variety.)

For each $\ell\not = p$, $H^1(X_{\bar\ff_q},\integ_\ell)$ (the dual of
the Tate module) is a free $\integ_\ell$-module of rank $2g$, endowed
with a symplectic pairing 
$\ang{\cdot,\cdot}_\lambda$ induced by the polarization.  The
Frobenius endomorphism $\varpi_{X/\ff_q}$ induces an element
$\gamma_{X/\ff_q,\ell} \in \gsp(H^1(X_{\bar\ff_q},\integ_\ell),\ang{\cdot,\cdot}_\lambda)$, and
thus an element of $G(\integ_\ell)$, well-defined up to conjugacy.
Moreover,
there is an equality of characteristic polynomials
$f_{\gamma_{X/\ff_q,\ell}}(T) = f_{X/\ff_q}(T)$. Simultaneously
considering all finite primes $\ell\not =p$, we obtain an adelic similitude  $\gamma_{[X,\lambda]} \in G(\aff^p_f)$.
(Alternatively one can, of course, directly consider the action of
$\varpi_{X/\ff_q}$ on $H^1(X_{\bar\ff_q},\hat\integ^p) =
\invlim{p\nmid n} H^1(X_{\bar\ff_q},\integ/n)$.)

Similarly, the crystalline cohomology group $H^1_\cris(X,\rat_q)$ is
endowed with an integral structure $H^1_\cris(X,\integ_q)$ and a
$\sigma$-linear endomorphism $F$. It determines, up to
$\sigma$-conjugacy, an element $\delta_{X/\ff_q}$ of $G(\rat_q)$ with multiplier $\eta(\delta_{X/\ff_q}) = p$. 

The $e^{th}$ iterate of $F$ is linear, and in fact $F^e$ is the endomorphism of
$H^1_\cris(X,\rat_q)$ induced by $\varpi_{X/\ff_q}$. 

Let $T_{[X,\lambda]}/\rat$ represent the automorphism group of $[X,\lambda]$ in the
category of abelian varieties up to $\rat$-isogeny.  Concretely, 
the polarization $\lambda$ induces a (Rosati) involution $(\dagger)$
on $\End(X)\tensor\rat$; and
for each $\rat$-algebra $R$, we have 
\[
T_{[X,\lambda]}(R) = \st{ \alpha \in (\End(X)\tensor R)\units : \alpha
  \alpha^{(\dagger)} \in R\units}.
\]

By Tate's theorem \cite{tateendff}, for $\ell\not = p$, $T_{[X,\lambda]}(\rat_\ell) \iso
G_{\gamma_{X/\ff_q,\ell}}(\rat_\ell)$, the centralizer of
$\gamma_{X/\ff_q,\ell}$, and \\ 
$T_{[X,\lambda]}(\rat_q) \iso
G_{\delta_{X/\ff_q}\sigma}(\rat_p)$, the twisted centralizer of
$\delta_{X/\ff_q}$ in $G(\rat_q)$.

A direct analysis of the effect of isogenies on the first cohomology
groups of abelian varieties then shows:

\begin{proposition}[\cite{kottwitz90}]
\label{proplk}
The weighted cardinality of the isogeny class of $[X,\lambda] \in \stk
A_g(\ff_q)$ is
\begin{equation}
\begin{split}
\label{eqlk}
\wnum I([X,\lambda],\ff_q)
 &= \vol(T_{[X,\lambda]}(\rat)\backslash T_{[X,\lambda]}(\aff_f)) \cdot
\int_{G_{\gamma_{X/\ff_q}}(\aff^p_f)\backslash G(\aff^p_f)}
\one_{G(\hat\integ^p_f)}(g\inv \gamma_{X/\ff_q} g)\, dg 
\\&\quad\cdot \int_{G_{\delta_{X/\ff_q}\sigma}(\rat_p)\backslash G(\rat_q)}
\one_{G(\integ_q) \operatorname{diag}(p, \cdots, p, 1, \cdots, 1)G(\integ_q)}(h\inv \delta_{X/\ff_q} h^\sigma)\, dh.
\end{split}
\end{equation}
\end{proposition}
In the orbital and twisted orbital integrals in \eqref{eqlk}, we choose the Haar measures on $G$ which assign volume $1$
to $G(\hat\integ^p)$ and to $G(\mathbb{Z}_q)$, respectively. 
The choice of measure on $T$ does not matter here, as long as the same measure is used to calculate the global volume. 
We define the specific measure on $T$ 
in  \S \ref{sec:global}.  
It coincides with the  canonical measure at all but finitely many places.

This formula appears in \cite[p.205]{kottwitz90}; see also \cite{kottwitz92} for its generalization to a much larger class of PEL Shimura varieties.  As in 
\cite[2.4]{achtergordon17}, the weighted cardinality accounts for the fact
that we have not introduced a rigidifying level structure, and thus our
objects admit nontrivial, albeit finite, automorphism groups.

\begin{rem}
\label{rem:hondatate}
Using Honda--Tate theory, one can (\cite[p.206]{kottwitz90},
\cite[p.422]{kottwitz92}) find $\gamma_{X/\ff_q,0} \in
G(\rat)$, well-defined up to $G(\bar\rat)$-conjugacy, such that
$\gamma_{X/\ff_q,0}$ and $\gamma_{X/\ff_q,\ell}$ are conjugate in
$G(\bar\rat_\ell)$.
Similarly, $\gamma_{X/\ff_q,0}$ and $\Norm
\delta_{X/\ff_q}$ are conjugate in $G(\bar\rat_q)$, where $\Norm$
denotes the norm map
\[
\xymatrix@R-2pc{
G(\rat_q) \ar[r] & G(\rat_q) \\
\alpha \ar@{|->}[r] & \alpha \alpha^\sigma \cdots \alpha^{\sigma^{e-1}}.
}
\]
In particular, the
characteristic polynomial of $\gamma_{X/\ff_q,0}$ is
$f_{X/\ff_q}(T)$.  In fact, by adjusting $\delta_{X/\ff_q}$ in its
twisted conjugacy class, we henceforth can and will assume that
\begin{equation}
  \label{E:normdelta}
  \Norm(\delta_{X/\ff_q})\in G(\rat_p) \subset G(\rat_q)
\end{equation}
\cite[p.206]{kottwitz82}.
  Then the group
variety $T_{[X,\lambda]/\ff_q}$ is isomorphic to the centralizer of $\Norm(\gamma_{X/\ff_q,0})$
in $G$.

It turns out that moreover, one can find a rational element $\gamma_0\in G(\Q)$ such that 
$\gamma_0$ is $G(\Q_\ell)$-conjugate to $\gamma_{X/\ff_q,\ell}$ for
every $\ell\neq p$ (see \cite[p.889]{kisin17}).  Consequently, in \eqref{eqlk} 
we  could replace $\gamma_{X/\ff_q}$ with a global object $\gamma_0$; but we will never use this fact in this paper.   
\end{rem}

In the remainder of this paper we fix a principally polarized
abelian variety $[X,\lambda]/\ff_q$ with commutative endomorphism ring
$\End(X)$.  (For example, any simple, ordinary abelian variety necessarily has a commutative endomorphism ring \cite[Thm.~7.2]{waterhouse:69}.) By Tate's theorem, the commutativity of $\End(X)$ is
equivalent to the condition that $T_{[X,\lambda]}$ is a maximal torus
in $G$.

To ease notation slightly, we will write $\delta_0$ and $T$ for
$\delta_{X/\ff_q}$ and $T_{[X,\lambda]}$, respectively. If $\ell$ is a
fixed, notationally suppressed prime, we will sometimes write
$\gamma_0$ for $\gamma_{X/\ff_q,\ell}$; by Remark \ref{rem:hondatate},
one may equally well let $\gamma_0$ be the image of some choice
$\gamma_{0}$ in $G(\rat)$ (though we will not be using it).

\subsection{Structure of the centralizer}
\label{subsec:torus}

For future use, we record some information about the centralizer $T =
T_{[X,\lambda]}$.  Recall that $X$ is a $g$-dimensional abelian variety with commutative 
endomorphism ring.  Then $T$ is a maximal torus in $G$,  and $K := \End(X)^0
= \End(X)\tensor\rat$ is a CM-algebra of degree $2g$ over $\rat$.
Then $K$ is isomorphic to a
direct sum $K \iso \oplus _{i=1}^t  K_i$ of CM fields, and the Rosati
involution on $\End(X)$ induces a positive involution $a\mapsto \bar
a$ on $K$, which in turn restricts to complex conjugation on each
component $K_i$.  Let
$K^+\subset K$ be the subalgebra fixed by the positive involution.
Then $K^+ \iso \oplus_{i=1}^t K^+_i$, where $K^+_i$ is the maximal
totally real subfield of $K_i$, and $[K^+:\rat] = g$.

In general, if $L$ is a field and $M/L$ is a finite \'etale algebra,
let $\res_{M/L}$ be Weil's restriction of scalars functor.
The norm map
$\Norm_{M/L}$ induces a map of tori $\res_{M/L}\gp_m \to \gp_m$, and the 
norm one torus is the kernel of this map:
\[
\xymatrix{
1 \ar[r]  & \resone_{M/L} \gp_m \ar[r] & \res_{M/L}\gp_m
\ar[r]^{\Norm_{M/L}} \ar[r]&  \gp_m \ar[r] & 1.
}
\]
With these preparations we have
\[
T^\der := T\cap G^\der \iso \res_{K^+/\rat}\resone_{K/K^+}\gp_m,
\]
and $T$ sits in the diagram
\begin{equation}
\label{eq:diagT}
\xymatrix{
1 \ar[r] & T^\der \ar[d]^\sim \ar[r] & T \ar@{^(->}[d] \ar[r] & \gp_m
\ar@{^(->}[d] \ar[r]& 1 \\
1 \ar[r]& \res_{K^+/\rat}\resone_{K/K^+}\gp_m  \ar[r] & \res_{K/\rat}
\gp_m \ar[r]^{\res_{K^+/\rat}\Norm_{K/K^+}} & \res_{K^+/\rat}\gp_m \ar[r] & 1.
}
\end{equation}
On points, we have
\begin{align*}
T(\rat) &= \st{a \in K\units: a \bar a \in \rat\units}\\
&= \st{(a_1, \cdots, a_t) \in \oplus K_i\units: \exists c \in
  \rat\units : a_i \bar a_i = c} \\
T^\der(\rat) &= \st{ a \in K\units: a \bar a =1} \\
&= \st{(a_1, \cdots, a_t) \in \oplus K_i\units: a_i \bar a_i = 1}.
\end{align*}

Let $\til T = T^\der \times \gp_m$.  It is not hard to write down an explicit isogeny $\alpha: \til T \to T$ and  a complementary isogeny $\beta:T \to \til T$ such that $\alpha \circ \beta$ is the squaring map.  We choose the maps which, on points, are given by
\[
\xymatrix@R-2pc{
T \ar[r]^\beta & \til T\ar[r]^\alpha & T \\
a \ar@{|->}[r] & (a \bar a\inv, a \bar a) \\
&(b,c) \ar@{|->}[r] & bc.
}
\]

\subsection{The Steinberg quotient}\label{sub:steinberg}

Recall that we have fixed a maximal split torus $T_\spl$ in $G$; let $W$ be the Weyl group
of $G$ relative to $T_\spl$.   Let $T_\spl^\der = T_\spl \cap G^\der$, and 
let
$A^\der=T_\spl^\der/W$ be the Steinberg quotient for the semisimple group
$G^\der$. It is isomorphic to the affine space of dimension $r-1=g$.

We let $\A_G=A^\der\times \gp_m$ be the analogue of the Steinberg quotient for the reductive group $G$, and define 
a map
\begin{equation}\label{eq:defSteinberg}
\xymatrix{
G \ar[r]^{\fc} & \A_G \\
\gamma \ar@{|->}[r]& (\tr(\gamma), \tr(\wedge^2\gamma), \cdots, \tr(\wedge^g\gamma), \eta(\gamma)).
}
\end{equation}
Note that $\eta(\gamma) = \tr(\wedge^{g+1}(\gamma))/\tr(\gamma)$; and if $\gamma \in G^\der \subset G$, then $\fc(\gamma) = (\fc^\der(\gamma),1)$, where $\fc^\der$ is the usual Steinberg map.

\subsection{Truncations}
\label{subsec:truncation}

Let $\ell$ be any finite prime (including $\ell = p$).  Let $\pi_n =
\pi_{\ell,n}: \integ_\ell \to \integ_\ell/\ell^n$ be the truncation
map.  
For any $\Z_\ell$-scheme $\cX$, we denote by $\pi_n^{\cX}$ the
corresponding map  
\[
\xymatrix{\pi_n^\cX: \cX(\integ_\ell) \ar[r] & \cX(\Z_\ell/\ell^n)}
\]
induced by $\pi_n$.  Given $S_n \subset \cX(\integ_\ell/\ell^n)$, we will often set
\[
\til S_n = \pi_n\inv(S_n).
\]

The projection maps $\pi^G_n$ extend to a somewhat larger set of similitudes.
Let
$M(\integ_\ell)$ be the set of symplectic similitudes which stabilize
the lattice $V\tensor \integ_\ell$;
\begin{align*}
M(\integ_\ell) &= \gsp(V\tensor \rat_\ell) \cap \End(V\tensor
\integ_\ell) \iso \gsp_{2g}(\rat_\ell)\cap \mat_{2g}(\integ_\ell).\\
\intertext{Inside this set, for each $d\ge 0$ we identify a subset}
M(\integ_\ell)_d &= \st{ A \in M(\integ_\ell) : \ord_\ell \det(A) \le  d}.
\\
\intertext{Finally, let us denote by $M(\integ_\ell/\ell^n)_d$ the set }
M(\integ_\ell/\ell^n)_d & = \st{ A \in M(\integ_\ell/\ell^n) : \ord_\ell \det(A) \le
  d}. 
\end{align*}
Note that $M(\integ_\ell)_0 = G(\integ_\ell)$, and in the last definition, 
the condition on the determinant is not vacuous even if $d
\gg n$, because it rules out the matrices of determinant zero.

With a certain amount of abuse, we introduce the following notion of ``$M(\integ_\ell)_d$-conjugacy'':
\begin{defn}\label{def:d-conj}
If $\gamma \in M(\integ_\ell)$, and in particular if $\gamma \in
G(\integ_\ell)$, 
we will write $\gamma\twiddle_{M(\integ_\ell)_d} \gamma_0$ if there
exists some $A \in M(\Z_\ell)_d$ such that
$A\gamma = \gamma_0A$.  

Similarly, if $\bar\gamma \in M(\integ_\ell/\ell^n)$, we write
$\bar\gamma\twiddle_{M(\integ_\ell/\ell^n)_d} \gamma_0$ if there
exists some $\bar A \in M(\integ_\ell/\ell^n)_d$ such that 
$\bar A\gamma = \pi_n(\gamma_0)\bar A$.  
\end{defn}

When $n$ is small relative to $d$, truncations of $M(\Z_\ell)_d$-
conjugate elements might not be $M(\integ_\ell/\ell^n)_d$-conjugate (since, e.g.,  all the elements $A\in M(\integ_\ell)$  satisfying $A\gamma = \gamma_0A$ might project to $0 \mod \ell^n$).  Of course, this does not happen when $n\gg d$. 
We also note that trivially, if $\gamma\twiddle_{M(\integ_\ell)_{d_0}} \gamma_0$ for some $d_0$, then $\gamma\twiddle_{M(\integ_\ell)_d} \gamma_0$ for all $d\ge d_0$.  The analogous statement holds for $\bar\gamma \in G(\integ_\ell/\ell^n)$ as long as $n\gg d$.

\subsection{Measures and integrals}
\label{subsec:measures}

As in \cite{achtergordon17}, we need to explicitly work out the
relationships between several different natural measures on the
$\ell$-adic points of varieties, especially groups and group
orbits.  The definitions introduced in \cite[\S 3]{achtergordon17} (where a little more historical perspective is briefly reviewed) 
go
through with minimal changes.  We recall the relevant notation here.

\begin{description}
\item[Serre-Oesterl\'e measure] In \cite[\S 3]{serre:chebotarev}, Serre observed that for a smooth $p$-adic 
submanifold $Y$ of $\Z_p^m$ of dimension $d$, there is a limit 
$\lim_{n\to \infty}\abs{Y_n}p^{-nd}$, where $Y_n$ is the reduction of $Y$ modulo $p^n$ (in our notation, $Y_n=\pi_n(Y)$).
Moreover, Serre pointed out that this limit can be understood as the \emph{volume} of $Y$ with respect to a certain measure, which is canonical. The definition of this measure for more general sets $Y$ was elaborated on by Oesterl\'e
\cite{oesterle} and by Veys \cite{veys:measure}. We refer to this measure as the Serre-Oesterl\'e measure, and denote it by $\mu^{\serre}$.

\item[Measures on groups] 

Once and for all, we fix the measure $\abs{dx}_\ell$ on the affine line $\aff^1_{\Q_\ell}$ to be the translation-invariant measure such that 
$\vol_{\abs{dx}_\ell}(\Z_\ell) =1$. Then 
there are two fundamentally different approaches to defining measure.  
The first is, for any smooth algebraic variety $\cX$ over $\Q_\ell$ with a non-vanishing top degree differential form 
$\omega$ on it, one gets the associated measure $\abs{d\omega}_\ell$ on $\cX(\Q_\ell)$.  
In particular, for a reductive group $G$, there is a canonical differential form $\omega_G$, defined in the greatest generality by Gross \cite{gross:motive}.  This gives a canonical measure $\abs{d\omega_G}_\ell$ on $G(\Q_\ell)$.
When $G$ is split over $\Q$, this measure has an alternative description using point-counting over the finite field (i.e., it coincides with Serre-Oesterl\'e measure $\mu_{G}^\serre$ defined above):
\begin{equation}\label{eq:can_meas}
\int_{G(\Z_\ell)}\abs{d\omega_G}_\ell = \frac{\#G(\Z/\ell)}{\ell^{\dim(G)}}.
\end{equation} 
  This observation is originally  due to A.~Weil \cite{weil:adeles}, and is actually built into his definition of integration on adeles.  Weil's classical observation is precisely what makes this paper possible. 

For groups, there is a second approach.  Start with a Haar measure and normalize it so that some given maximal subgroup has volume $1$. The choice of a ``canonical'' compact subgroup in this approach could lead to very interesting considerations (and is one of the main points of  \cite{gross:motive}), but  
in our situation only one easy case is needed.  
For $G(\Q_\ell)$, the relevant maximal subgroup is $G(\Z_\ell)$; 
we denote such a Haar measure on $G(\Q_\ell)$ by $\mu_G^\can$.

\item[Geometric measure on orbits] 
 This is a measure
  constructed in \cite{langlands-frenkel-ngo} on a fiber of the Steinberg map
  $\fc: G \to \A_G$.  Let $\omega_G$
  be a volume form on $G$, and let $\omega_A$ be the volume form
  $\bigwedge dx_i \wedge \frac{dx}{\abs{x}}$ on $\A_G \iso \aff^{\rank(G)-1} \times \gp_m$.
  On the fiber 
  $\fc\inv(\fc(\gamma))$,
  factor $\omega_G$ as 
\[
\omega_G = \omega^\geom_{\fc(\gamma)}\wedge \omega_A;
\]
integrating $\abs{\omega^\geom_{\fc(\gamma)}}$ defines a measure
$\mu^\geom$ on $\fc\inv(\fc(\gamma))$.

\end{description}

Suppose $\phi$ is a locally constant compactly supported function on
$G(\rat_\ell)$.  Recall 
the family $\gamma_{X/\F_q, \ell}$  (and $\delta_0$), whose centralizers are the sets of $\Q_\ell$-points of 
the algebraic torus $T:=T_{[X, \lambda]}$. 
We
use two different measures on the orbit $\orbit{G(\rat_\ell)}{\gamma_{X/\F_q, \ell}} \iso
T_{[X, \lambda]}(\rat_\ell)\backslash G(\rat_\ell)$ to define an
integral.  When $\ell$ is fixed, we will often denote the element $\gamma_{X/\F_q, \ell}$ by $\gamma_0$; we define 
$\mu^\tama_{\gamma_0}$ as the quotient measure 
$\mu_{G}^\can/\mu_T^\tama$, where $\mu_T^\tama$ is the Tamagawa measure on $T$ defined below in \S\ref{sub:tama},  and let $\mu^\geom_{\gamma_0}$ be the
geometric measure reviewed above.  (Since the orbit of $\gamma_0$
is an open subset of $\fc\inv(\fc(\gamma_0))$, the restriction of the geometric measure from $\fc\inv(\fc(\gamma_0))$ to the orbit makes sense.)  Then for $\bullet \in
\left\{\tama, 
\geom\right\}$, set
\[
O_{\gamma_0}^\bullet := \int_{T(\rat_\ell)\backslash G(\rat_\ell)}
\phi(g\inv \gamma_0 g)d\, \mu^\bullet_{\gamma_0}.
\]

\section{Conjugacy}\label{sec:conj}

\subsection{Integral conjugacy}
To 
relate the right-hand side of \eqref{eqlk} to the ratios $\nu_\ell$ of \eqref{eqdefnuintronaive}, we
interpret the orbital integral as  the volume of the intersection of the
$G(\rat_\ell)$-orbit of $\gamma_0$ with $G(\integ_\ell)$.
For almost all $\ell$, $G(\integ_\ell)\cap \orbit{G(\rat_\ell)}{\gamma_0} = \orbit{G(\integ_\ell)}{\gamma_0}$:

\begin{lemma}
\label{lem:gqvsgz}
Suppose $\gamma_0 \in G(\integ_\ell)$ and $\ell \nmid
D(\gamma_0)$.  If $\gamma \in G(\integ_\ell)$, then
\[
\gamma \twiddle_{G(\rat_\ell)}\gamma_0 \Longleftrightarrow \gamma
\twiddle_{G(\integ_\ell)}\gamma_0.
\]
\end{lemma}

\begin{proof}
The hypothesis on $\gamma_0$ implies that the centralizer
$G_{\gamma_0}$ 
 is a smooth torus over $\integ_\ell$, and thus
the transporter from $G_\gamma$ to $G_{\gamma_0}$ is smooth over
$\integ_\ell$ (e.g., \cite[Prop.\ 2.1.2]{conradsga3}).

Since $\gamma$ and $\gamma_0$ are conjugate in $G(\rat_\ell)$, they
have the same characteristic polynomial, and thus their reductions
$\bar\gamma_0 = \pi^G_1(\gamma_0)$ and $\bar\gamma = \pi^G_1(\gamma)$
are stably conjugate in $G(\integ_\ell/\ell)$.  By Lang's theorem,
$\bar\gamma_0$ and $\bar\gamma$ are conjugate in $G(\integ_\ell/\ell)$; by
smoothness of the transporter scheme, $\gamma$ and $\gamma_0$ are
conjugate in $G(\integ_\ell)$.
\end{proof}

If $\bar\gamma_0$ is \emph{not} regular, then the set  $G(\integ_\ell)\cap
\orbit{G(\rat_\ell)}{\gamma_0}$ generally consists of several different
$G(\integ_\ell)$-orbits.  Nonetheless, the number of distinct orbits
is bounded; and membership in $\orbit{G(\rat_\ell)}{\gamma_0}$ can be 
detected at a finite truncation level. 

\begin{lemma}
  \label{lem:hensel}
Suppose $\gamma_0 \in G(\integ_\ell)$ is regular semisimple.  There
exists an integer $e=e(\gamma_0)$ such that, if $n\gg0$  and $d>e$, 
then for $\gamma \in
G(\integ_\ell/\ell^n)$, the following conditions are equivalent: 
\begin{enumerate}
\item $\gamma \twiddle_{M(\integ_\ell/\ell^n)_d} \gamma_0
\bmod \ell^n$, and 
\item  there exists some $\til \gamma \in
G(\integ_\ell)$ such that $\til \gamma \bmod \ell^n = \gamma$ and
$\til\gamma\twiddle_{G(\rat_\ell)} \gamma_0$.
\end{enumerate}
The statement is also true with $G(\Z_\ell)$  replaced with  $M(\integ_\ell)$ everywhere.
\end{lemma}

\begin{proof} 
We prove the original statement.
 
The intersection of $G(\Z_\ell)$ with the $G(\Q_\ell)$-orbit of 
$\gamma_0$ is a finite union of $G(\Z_\ell)$-orbits, since it is compact  (recall that $\gamma_0$ is regular semisimple) and the $G(\Z_\ell)$-orbits are open in this intersection; let $g_1,...,  g_s$ be representatives of these orbits, and let $A_i\in G(\Q_\ell)$ be elements  satisfying $A_i g_i A_i^{-1}=\gamma_0$, so that $A_ig_i = \gamma_0 A_i$.  We clear denominators; for each $i$, let $X_i\in M(\Z_\ell)$ be a scalar multiple of $A_i$.  Then $X_ig_i = \gamma_0X_i$, and we set
$$e(\gamma_0)=\max_{i\in\{1,.. ,s\}}\{\abs{\ord(\det X_i)}\}.$$

Now, suppose $n > 2d(\gamma_0)$, where  $d(\gamma_0)$ is the valuation of the discriminant of $\gamma_0$, $e\ge e(\gamma_0)$, and  $n \gg e$. 
We want to prove that with these assumptions, an element $\gamma \in G(\integ_\ell/\ell^n)$  satisfies 
$$\gamma \twiddle_{M(\integ_\ell/\ell^n)_e} \pi_n(\gamma_0)$$ if and only if there exists a lift  $\til \gamma \in
G(\integ_\ell)$ such that $\pi_n(\til \gamma)
= \gamma$ and
$\til\gamma\twiddle_{G(\rat_\ell)} \gamma_0$.

One direction is easy: suppose there exists $\til \gamma \in
G(\integ_\ell)$ such that $\til \gamma \bmod \ell^n = \gamma$ and
$\til\gamma\twiddle_{G(\rat_\ell)} \gamma_0$. 
Then there exists $i\in \{1, \dots, s\}$ such that $\til \gamma \twiddle_{G(\Z_\ell)} g_i$. 
Therefore there exists $Y\in G(\Z_\ell)$ such that $Y\til \gamma = g_i Y$. Recall that as above, there exists $X_i\in M(\Z_\ell)$ such that 
$X_i g_i = \gamma_0 X_i$. 
Then $Z:=\pi_n^{M}(X_iY)$ lies in $M(\Z_\ell/\ell^n)_e$ and satisfies the condition 
$Z \gamma=\pi_n(\gamma_0)Z$.

The other direction  is a special case of Hensel's lemma. Since Hensel's lemma in this generality, though well-known, is surprisingly hard to find in the literature, we provide a detailed explanation with references.

For each $n$, let
\[
R_{\gamma_0}(\integ_\ell/\ell^n) = \st{ (A, \gamma): A \in M(\integ_\ell/\ell^n),
  \gamma \in G(\integ_\ell/\ell^n), A\gamma = \pi_n^G(\gamma_0) A} \subset
M(\integ_\ell/\ell^n) \times G(\integ_\ell/\ell^n),
\]
where $\pi_n^G$ is the projection from \S\ref{subsec:truncation}.
This is a system of $(2g)^2$ equations in $8g^2$ variables (namely, the matrix entries of $A$ and $\gamma$).
Now Hensel's lemma as stated in 
\cite[III.4.5., Corollaire 3, p.271]{Bourbaki:commalg}
applies directly, as follows.  
Let $n(\gamma_0)$ be the valuation of the minor formed by the first $(2g)^2$ columns of the Jacobian matrix of this system of equations at 
$\gamma_0$.
By Hensel's lemma,  if $n >
2n(\gamma_0)$ and $(A,\gamma) \in R_{\gamma_0}(\integ_\ell/\ell^n)$, then there
exists some $\til \gamma \in G(\integ_\ell)$ such that $\pi_n(\til\gamma)
= \gamma$ and $\til \gamma \twiddle_{M(\integ_\ell)}
\gamma_0$.

Since the core argument simply relies on the solvability, via Hensel's lemma, of a system of equations over $\integ_\ell$, it is also valid if $G(\integ_\ell)$ is replaced by $M(\integ_\ell)$.
\end{proof}

\begin{rem}
We observe (though we do not need this observation in this paper) that $n(\gamma_0)$ in fact equals the valuation of the discriminant of 
$\gamma_0$, e.g. by  the argument provided  in \cite[\S 7.2]{kottwitz:clay}.
\end{rem}

For $\gamma_0 \in G(\integ_\ell)$, let
\begin{equation}
  \label{eq:cdn}
C_{(d,n)}(\gamma_0)=
 \st{\gamma \in G(\integ_\ell/\ell^n) : \gamma \twiddle_{M(\integ_\ell/\ell^n)_d}
\gamma_0}.
\end{equation}
If $d = 0$, this coincides with the usual conjugacy class of
$\pi_n(\gamma_0)$.  As in Section \ref{subsec:truncation}, let
$\til C_{(d,n)}(\gamma_0)=({\pi_n^G})^{-1}(C_{(d,n)}(\gamma_0))$ be the set of lifts of
elements of $C_{(d,n)}(\gamma_0)$ to $G(\integ_\ell)$.

We also extend this notation to elements $\gamma_0 \in
M(\integ_\ell)$:
\[
C_{(d,n)}(\gamma_0) = \st{\gamma \in M(\integ_\ell/\ell^n) : \gamma
  \twiddle_{M(\integ_\ell/\ell^n)_d} 
\gamma_0}.
\]
(If $\gamma_0 \in G(\integ_\ell) \subset M(\integ_\ell)$, the two
notions coincide and thus there is no ambiguity.)

\begin{corollary}
\label{cor:McapturesG}

  \begin{alphabetize}
    \item Suppose $\gamma_0 \in G(\integ_\ell)$.  
      There exists $d = d(\gamma_0)$ such that, if $n\gg 0$,
      then
      \begin{align*}
        C_{(d,n)}(\gamma_0) &= \pi_n(G(\integ_\ell)\cap \orbit{G(\rat_\ell)}{\gamma_0}).\\
        \intertext{Moreover,}
        G(\integ_\ell)\cap \orbit{G(\rat_\ell)}{\gamma_0} &= \bigcap_{n\ge 0} \til C_{(d,n)}(\gamma_0).
      \end{align*}
      \item Suppose $\gamma_0 \in M(\integ_\ell)$.  There exists $d =
        d(\gamma_0)$ such that, if $n\gg 0$, then
      \begin{align*}
        C_{(d,n)}(\gamma_0) &= \pi_n(M(\integ_\ell)\cap
                              \orbit{G(\rat_\ell)}{\gamma_0}). \\
        \intertext{Moreover,}
      M(\integ_\ell)\cap \orbit{G(\rat_\ell)}{\gamma_0} &= \bigcap_{n\ge 0} \til C_{(d,n)}(\gamma_0).
\end{align*}
\end{alphabetize}
\end{corollary}

\begin{proof}
This is a direct consequence of Lemma \ref{lem:hensel}.
\end{proof}

\subsection{Stable (twisted) conjugacy}\label{subsec:twisted_conj}

In this section, we further assume that $[X,\lambda]$ is a 
principally-polarized abelian variety with 
commutative endomorphism ring for which $1/2$ is not a slope of the
Newton polygon of $X$.  (Again, any ordinary simple principally polarized
abelian variety satisfies these hypotheses.)

Recall the definition of $K$ and $K^+$, as well as the
discussion of $T$, from Section \ref{subsec:torus}.  By a prime of $K$
(or $K^+$) lying over $p$ we mean a prime $\idp$ of some $K_i$
(respectively, $K^+_i$) lying over $p$,
and we write $K_\idp$ for $K_{i,\idp}$.  With this convention, we then have $K\tensor \rat_p \iso \oplus K_\idp$.

\begin{lemma}
\label{lem:pplussplits}
Let $\idp^+$ be a prime of $K^+$ lying over $p$.  Then $\idp^+$ splits
in $K$.
\end{lemma}

\begin{proof}
This is standard.  We work in the category of $p$-divisible groups up
to isogeny.  Then $X[p^\infty]$ has height $2g$, and comes equipped
with an action by $K \tensor_\rat \rat_p$.

Corresponding to the decomposition $K^+\tensor_\rat \rat_p
\iso \oplus _{\idp^+|p} K^+_{\idp^+}$ we have the decomposition
\[
X[p^\infty] = \oplus X[\idp^{+\,\infty}].
\]
Moreover, $X[\idp^{+\,\infty}]$ is a $p$-divisible group of height
$2[K^+_{\idp^+}:\rat_p]$, and self-dual (because $\idp^{+} \calo_K$ is
stable under the Rosati involution).  We now fix one $\idp^+$, and
show that it must split in $K$.

Since $K^+_{\idp^+}$ is a field
(and not just a $\rat_p$-algebra) of dimension $\half 1
\operatorname{ht}(X[\idp^{+\,\infty}])$, $X[\idp^{+\,\infty}]$ has at most two slopes.
Since by hypothesis $1/2$ is not a slope of $X$, $X[\idp^{+\,\infty}]$
has exactly two slopes, say $\lambda = a/b$ and $1-a/b$, where
$\gcd(a,b) = 1$.   Let $m$ be the multiplicity of $\lambda$ as a slope
of $X[\idp^{+\,\infty}]$; then $mb = [K^+_{\idp^+}:\rat_p]$.  The
endomorphism algebra of $X[\idp^{+\,\infty}]$ (again, in the category
of $p$-divisible groups up to isogeny) is isomorphic to
\[
\End(X[\idp^{+\,\infty}])^0 \iso \mat_m(D_\lambda)\oplus
\mat_m(D_{1-\lambda}),
\]
where $D_\lambda$ is the central simple $\rat_p$-algebra with Brauer
invariant $\lambda$.  In particular, any subfield $L$ of
$\End(X[\idp^{+\,\infty}])^0$ satisfies $[L:\rat_p] \le mb =
[K^+_{\idp^+}:\rat_p]$.   Since $K\tensor_{K^+}K^+_{\idp^+}$ acts on
$X[\idp^{+\,\infty}]$, we must have $K\tensor_{K^+}K^+_{\idp^+} \iso
K^+_{\idp^+} \oplus K^+_{\idp^+}$, as claimed.
\end{proof}

\begin{corollary}
\label{cor:Tderp}
We have
\begin{align*}
T_{\rat_p}^\der  \iso \oplus_{\idp^+} \res_{K^+_{\idp^+}/\rat_p}\gp_m.
\end{align*}
\end{corollary}

\begin{proof}
Since $T^\der = \res_{K^+/\rat} \resone_{K/K^+}\gp_m$, using Lemma \ref{lem:pplussplits}
we find
\begin{align*}
T^\der_{\rat_p} &= \res_{K^+\tensor\rat_p/\rat_p}
                  \resone_{K\tensor\rat_p/K^+\tensor\rat_p} \gp_m \\
&=\iso \oplus_{\idp^+|p} \res_{K^+_{\idp^+}/\rat_p}
  \resone_{K\tensor_{K^+}K^+_{\idp^+}/K^+_{\idp^+}}\gp_m.
\end{align*}
If $L$ is any field then $\res_{L\oplus L/L}\gp_m \iso
\gp_{m,L}\oplus \gp_{m,L}$; the norm map $\res_{L\oplus L/L}\gp_m \to
\gp_m$ is given by multiplication of components; and so
$\resone_{L\oplus L/L}\gp_m$ is isomorphic to $\gp_{m,L}$, where the
latter is embedded in the former via $(\operatorname{id},\operatorname{inv})$.
\end{proof}

Recall that the twisted conjugacy class of $\Norm(\delta_0)$ is defined over
$\rat_p$, and consequently contains a
$G(\rat_p)$-rational representative (\cite[p.799,
Thm. 4.4]{kottwitz82}); and we have adjusted $\delta_0$ in its twisted
conjugacy class to ensure that $\gamma_0 := \Norm(\delta_0) \in
G(\rat_p)$ \eqref{E:normdelta}.

\begin{lemma}
\label{lem:singleclass}
The stable conjugacy class of $\gamma_0$ consists of a single
conjugacy class, and the stable $\sigma$-conjugacy class of
$\delta_0$ consists of a single $\sigma$-conjugacy class.
\end{lemma}

\begin{proof}
To prove the first claim, it suffices (by \cite[p.788]{kottwitz82}) to
show that $H^1(\rat_p, T)$ vanishes.  By taking the long exact sequence of cohomology of the top row of \eqref{eq:diagT}, and then invoking Hilbert 90 and
Corollary \ref{cor:Tderp}, we find that $H^1(\rat_p,T)$ does in fact
vanish.

For the second claim, it similarly suffices to show that the first
cohomology of the twisted centralizer $G_{\delta_0,\sigma}$ vanishes
\cite[p.805]{kottwitz82}.  However, the twisted centralizer of an
element is always an inner form of the (usual) centralizer of its norm
\cite[Lemma 5.8]{kottwitz82}.  In our case, the centralizer $T =
G_{\gamma_0}$ is a torus, and thus admits no nontrivial inner forms.
We conclude again that $H^1(\rat_p, G_{\delta_0,\sigma})$ is trivial.
\end{proof}

\section{Ratios}
\label{sec:ratios}

\subsection{Definitions}

For $\ell\not =p$, we define a local ratio $\nu_\ell([X,\lambda])$
designed to measure the extent to which the conjugacy class of
$\gamma_{X_0/\ff_q}$, as an element of $G(\integ_\ell/\ell)$, is more or
less prevalent than a randomly chosen conjugacy class. More precisely, to measure this probability, 
we consider the finite group $G(\Z_\ell/\ell^n)$ for sufficiently large $n$, and recall that our notion of ``conjugacy'' in this group is not the usual conjugacy but the relation $\sim_{M(\Z_\ell/\ell^n)_e}$ defined above in \S\ref{subsec:truncation}.   
For $\ell=p$, the element $\gamma_{X_0/\ff_q}$ is not in $G(\Z_p)$, and we use 
$M(\Z_p)$ instead; but this has no effect on the definition since our notion of ``conjugacy'' 
in $G(\Z_p/p^n)$ already uses $M(\Z_p)$.   

Recall the definition of $C_{(d,n)}(\gamma_0)$
from \eqref{eq:cdn}, and that
$C_{n}(\gamma_0) : = C_{(0,n)}(\gamma_0)$ is the actual
conjugacy class of $\pi_n(\gamma_0)$ in $G(\integ_\ell/\ell^n)$.

\begin{defn}\label{def:gekratio}
For each finite place $\ell$, including $\ell=p$, using the shorthand  $\gamma_0:=\gamma_{X/\F_q, \ell}\in M(\Z_\ell)$, set
\begin{equation}
\label{eq:defgekeler}
\nu_\ell([X,\lambda]) = \nu_\ell([X,\lambda],\ff_q) =
\lim_{d\ra\infty}\lim_{n\ra\infty} 
\frac{\#C_{(d,n)}(\gamma_0)}{\# G(\integ_\ell/\ell^n)/\#\A_G(\integ_\ell/\ell^n)}.
\end{equation}
At infinity, define 
\begin{equation}\label{eq:gekinf}
\nu_\infty([X,\lambda]) = \nu_\infty([X,\lambda],\ff_q) = \frac{\abs{D(\gamma_0)}_{\infty}^{1/2}}{(2\pi)^{g}}
\end{equation}
where $\abs{\cdot}_\infty$ is the usual real absolute value.
\end{defn}

\begin{rem} So far we have avoided using the fact that there exists a rational element $\gamma_0\in G(\Q)$  as in Remark \ref{rem:hondatate}, and treated $\gamma_0$ as an element of $G(\A_f)$. We can continue doing so, and  then for \eqref{eq:gekinf} simply \emph{define} the archimedean absolute vaue of its discriminant by 
$\abs{D(\gamma_0)}_{\infty}:=\prod_\ell\abs{{D(\gamma_{X/\F_q, \ell})}}_\ell^{-1}$.  
\end{rem}

The ratios stabilize for large enough $d$ and $n$, and thus the limits are not, strictly speaking, necessary.  
In fact, for $\ell\neq 2, p$ and not dividing the discriminant of $\gamma_0$, the ratios stabilize right away, at $d=0$ and $n=1$, as the next two lemmas show.

\begin{lemma}
\label{lem:cdncn}
If $\ell \nmid D(\gamma_0)$ and $n\gg 0$, 
then $C_{(d,n)}(\gamma_0) =
C_{n}(\gamma_0)$.
\end{lemma}

\begin{proof}
Recall that all our notation assumes that $\ell$ is fixed.
Clearly $C_{n}(\gamma_0) \subseteq C_{(d,n)}(\gamma_0)$.  Conversely,
suppose $n$ is sufficiently large as in Lemma \ref{lem:hensel} and
that $\gamma \in C_{(d,n)}(\gamma_0)$.  Then there exists some $\til\gamma
\in \pi_n\inv(\gamma)$ such that $\til \gamma \twiddle_{G(\rat_\ell)}
\gamma_0$.  By Lemma \ref{lem:gqvsgz},
$\til\gamma\twiddle_{G(\integ_\ell)} \gamma_0$, and so $\gamma
\twiddle_{G(\integ_\ell/\ell^n)} \gamma_0$.
\end{proof}

\begin{lemma}
\label{lem:gekelerunram}
If $\ell\nmid D(\gamma_0)$, $\ell\neq 2$,  and $d\gg_{\gamma_0}0$, then for  $n\ge 1$,
\begin{equation}
\label{eq:calcgekelerunram}
\frac{\#C_{(d,n)}(\gamma_0)}{\#G(\integ_\ell/\ell^n)/\#\A_G(\integ_\ell/\ell^n)} =
\frac{\#\st{\gamma \in G(\integ_\ell/\ell): \gamma\twiddle
    \gamma_0}}{\#G(\integ_\ell/\ell)/\#\A_G(\integ_\ell/\ell)}.
\end{equation}
\end{lemma}

\begin{proof}
By Lemma \ref{lem:cdncn}, the left-hand side of \eqref{eq:calcgekelerunram} is
\[
\frac{\#G(\integ_\ell/\ell^n)/\#G_{\gamma_0}(\integ_\ell/\ell^n)}{\#G(\integ_\ell/\ell^n)/\#\A_G(\integ_\ell/\ell^n)}
= \frac{\#\A_G(\integ_\ell/\ell^n)}{\#G_{\gamma_0}(\integ_\ell/\ell^n)}.
\]
Since $\pi_1(\gamma_0)$ is regular, the centralizer $G_{\gamma_0}$ is
smooth over $\integ_\ell$ of relative dimension $g+1$.  Since the same
is true of the scheme $\A_G/\integ_\ell$, the result now follows.
\end{proof}

\subsection{From ratios to integrals}

Fix a prime $\ell$ (possibly $\ell=p$ or $\ell=2$).
(In this subsection, as above, all quantities depend on this notationally suppressed prime.)  Recall \eqref{eq:defSteinberg} the canonical map $\fc:G \to A$ from $G$ to its
Steinberg quotient. The fibres of this map over regular points are stable orbits of regular semisimple elements. 
Define a system of neighbourhoods of $\fc(\gamma_0)$ inside
$\A_G(\integ_\ell)$ by
\[
\til U_n(\gamma_0) = \til{\pi_n^{\A_G}(\fc(\gamma_0))} = (\pi_n^{\A_G})\inv(\pi_n^{\A_G}(\fc(\gamma_0))).
\]

In other words, 
\begin{equation*}
\til U_n(\gamma_0)  =\{a=(a_1, \dots a_{g},\eta)\in \A_G(\Z_\ell) \mid a_i \equiv a_i(\gamma_0) \bmod \ell^n, \eta \equiv \eta(\gamma_0)\bmod \ell^n \}.
\end{equation*}

These definitions  and \eqref{eq:cdn} are summarized by the diagram
\begin{equation}
\label{eq:reduceSteinberg}
\xymatrix{
 \til C_{(d,n)}(\gamma_0) \subset M(\Z_\ell) \ar[d]<6.5ex>^{\pi_n^M} \ar[r]^{\fc}&  \A_G(\Z_\ell)
 \supset \til U_n(\gamma_0) \ar[d]<-5ex>^{\pi_n^{\A_G}} \\
 C_{(d,n)}(\gamma_0) \subset M(\Z_\ell/\ell^n) \ar[r]^{\fc_n} & \A_G(\Z/\ell^n)\ni  \pi_n^{\A_G}(\fc(\gamma_0)),
}
\end{equation}
where $\fc_n:G(\Z/\ell^n)\to \A_G(\Z/\ell^n)$ is the map sending an element to the coefficients of its characteristic polynomial $\bmod\ \ell^n$. 
The diagram of maps commutes since reduction $\bmod\ \ell^n$ is a ring homomorphism, and the map $\fc$ is polynomial in the matrix entries of $\gamma$.  (The diagram of subsets need not commute, though.)  We also note that when $\ell\neq p$, the sets 
$\til C_{(d,n)}(\gamma_0)$ and $C_{(d,n)}(\gamma_0)$ are contained in $G(\Z_\ell)$ and $G(\Z_\ell/\ell^n)$, respectively, since
$\ord_\ell(\det(\gamma_0))=0$ in this case, and this is also true for all elements that are congruent to any conjugate of 
$\gamma_0$.

By definition of the geometric measure, for any open subset $B\subset
G(\Z_\ell)$ we have
\begin{equation}
\label{eq:st.vol}
\vol_{\mu^\steinberg}(B\cap \fc^{-1}(\fc(\gamma_0)))=\lim_{n\to \infty}
\frac{\vol_{{\abs{d\omega_G}}}(\fc^{-1}(\til U_n(\gamma_0))\cap B)}{\vol_{\abs{d\omega_A}}(\til U_n(\gamma_0))}.
\end{equation}
Recall that each stable orbit $\fc^{-1}(\fc(\gamma))$ is a finite disjoint union of rational orbits.  Each rational orbit being an open subset of the stable orbit, we may and do define geometric measure on each rational orbit, by restriction.

In simple terms, the sets $\til U_n$ form a system of neighbourhoods of the point $\fc(\gamma_0)\in \A_G$; the set
$\til C_{(d,n)}(\gamma_0)$ can be thought of as the intersection of a neighbourhood of the orbit of $\gamma_0$ with $G(\Z_\ell)$; 
the set $\fc^{-1}(\fc(\gamma_0))$ is the stable orbit of $\gamma_0$.  The following lemma gives the precise relationships between all these sets.

\begin{lemma}\label{lem:U_n}
\begin{alphabetize}
\item Let $\ell\neq p$. For large enough $d$ and $n$ (depending on $\gamma_0$), we have 
\begin{equation}
\label{eq:stableorbit}
\fc^{-1}(\til U_n(\gamma_0)) \cap G(\Z_\ell) =
\bigcup_{\gamma'\twiddle_{G(\overline{\Q}_\ell)} \,\gamma_0} \til C_{d,n}(\gamma'),
\end{equation}
where $\gamma'$ runs over a set of  representatives of $G(\Q_\ell)$-conjugacy classes in the stable conjugacy class of $\gamma_0$ whose $\Q_\ell$-orbits intersect 
$G(\Z_\ell)$, so that we may take the elements $\gamma'$ to lie in $G(\Z_\ell)$.   
\item When $n$ is sufficiently large (depending on $\gamma_0$), the sets $\til C_{d,n}(\gamma')$ above are disjoint. 
\item Let $\mu_G^{\serre}$ be the Serre-Oesterl\'e measure on $G(\Q_p)\cap M(\Z_p)$, viewed as a submanifold  of $M(\Z_p)$. Then  $\vol_{\mu_G^\serre}(\til C_{d,n}(\gamma_0))= \ell^{-n\dim(G)}\# C_{d,n}(\gamma_0)$; in particular, if 
$\ell\neq p$,
$\vol_{\abs{d\omega_G}}(\til C_{d,n}(\gamma_0))= \ell^{-n\dim(G)}\# C_{d,n}(\gamma_0).$
\end{alphabetize}
\end{lemma}
\begin{proof}
{\bf (a).}
This is an easy consequence of the fact that 
two regular semisimple elements of $G(\Q_\ell)$ are stably conjugate if and only if their characteristic polynomials coincide. 
In our notation,  
$$\fc^{-1}(\fc(\gamma_0)) =
\sqcup_{\gamma'\twiddle_{G(\overline{\Q_\ell})}
  \,\gamma_0} {(\orbit{G(\Q_\ell)}{\gamma'})}, 
$$   
where $\orbit{G(\Q_\ell)}{\gamma'}$ denotes the rational conjugacy class of $\gamma'$ in $G(\Q_\ell)$ as before.
Now, we will describe both the left-hand side and the right-hand side of \eqref{eq:stableorbit} as: 
the set of elements  $\gamma\in G(\Z_\ell)$ whose characteristic polynomial is congruent to that of $\gamma_0$ $\bmod\ \ell^n$. 
Indeed, on the left-hand side, by definition,  $\gamma\in \fc^{-1}(\til U_n(\gamma_0))\cap G(\Z_\ell)$ if and only if 
$\pi_n^{\A_G}(\fc(\gamma))=\pi_n^{\A_G}(\fc(\gamma_0))$.  By the commutativity of \eqref{eq:reduceSteinberg}, this is equivalent to 
$\fc_n(\pi_n^G(\gamma))=\fc_n(\pi_n^G(\gamma_0))$, i.e.,  the characteristic polynomials of $\gamma$ and $\gamma_0$ are congruent $\bmod \ell^n$.
 On the right-hand side,  given $\gamma'\in G(\Z_\ell)$,  by Lemma \ref{lem:hensel}, for $d$ and $n$ large enough 
\footnote{\emph{Large enough}  depends on $\gamma'$, but only through its discriminant.  Since stably conjugate elements have the same discriminant, ultimately this only depends on $\gamma_0$.}, we have that $\gamma\in \til C_{(d,n)}(\gamma')$ if and only if 
there exists $\gamma'' \in G(\Z_\ell)$ such that $\gamma'' \equiv \gamma' \bmod \ell^n$ and $\gamma''$ is $G(\Q_\ell)$-conjugate to $\gamma$. 
Taking the union of these sets as $\gamma'$ runs over the set of integral representatives of $G(\Q_\ell)$-conjugacy classes in the stable class of $\gamma_0$, we obtain the set of all elements $\gamma\in G(\Z_\ell)$ that are congruent modulo $\ell^n$ to an element of $G(\Z_\ell)$ that is stably conjugate to $\gamma_0$, i.e., to an element having the same characteristic polynomial as $\gamma_0$.  This means that 
$\fc_n(\pi_n^G(\gamma))=\fc_n(\pi_n^G(\gamma_0))$, which completes the proof of the first statement.

{\bf (b).} Since the orbits of regular semisimple elements are closed in the $\ell$-adic topology, distinct orbits have disjoint neighbourhoods. 

{\bf (c).} The map $\pi_n^M: \til C_{(d,n)}(\gamma_0) \to C_{(d,n)}(\gamma_0)$ is surjective, so $\til C_{(d,n)}(\gamma_0)$ can be thought of as a disjoint union of fibres of $\pi_n^M$. Since $M$ is a smooth scheme over $\Z_\ell$, each fibre of $\pi_n^M$ has volume    
$\ell^{-n\dim(G)}$ with respect to the measure $\mu^\serre$ (cf. \cite{serre:chebotarev}). The first statement follows. Moreover, as discussed above in \S\ref{subsec:measures}, on $G(\Z_\ell)$, the measures $\mu^\serre$ and $\mu_{\abs{\omega_G}}$ coincide. For $\ell\neq p$, we have $\til C_{(d,n)}(\gamma_0)\subset G(\Z_\ell)$, which completes the proof.  
\end{proof}

Recall that $\phi_0$ is the characteristic function of $G(\integ_\ell)$.

\begin{corollary}
Let $\ell\neq p$. Then there exists $d(\gamma_0)$ such that for $d\ge d(\gamma_0)$
\label{lem:stein_oireprise}
$$O^\steinberg_{\gamma_0}(\phi_0)=\lim_{n\to
  \infty}\frac{\vol_{{\abs{d\omega_G}}}(\til
  C_{(d,n)}(\gamma_0))}{\vol_{\abs{d\omega_{A}}}(\til U_n(\gamma_0))}.$$
\end{corollary}

\begin{proof}
The orbital integral, by definition, calculates the volume of the set of integral points in the rational orbit of $\gamma_0$, with respect to the geometric measure on the orbit. 
Using Lemma \ref{lem:U_n}(a)-(b)
we write $\fc^{-1}(\til U_n(\gamma_0))\cap G(\Z_\ell) = \sqcup_{\gamma'} \til C_{d,n}(\gamma')$, where $\gamma'$ are as in that lemma, with $\gamma_0$ being one of the elements $\gamma'$. 
The union on the right-hand side of \eqref{eq:stableorbit} is a disjoint union of neighbourhoods of the individual orbits, intersected with $G(\Z_\ell)$. 
The statement follows from the equality \eqref{eq:st.vol}, applied to the set $B:=\til C_{d, n}(\gamma_0)$.  
\end{proof}

\begin{corollary}
\label{cor:gek2stein}
For $\ell\not =p$, the Gekeler ratio \eqref{eq:defgekeler} is related to the geometric orbital integral by 
\[
\nu_\ell([X,\lambda]) = \frac{\ell^{\dim(G^\der)}}{\#G^\der(\integ_\ell/\ell)}O^\steinberg_{\gamma_0}(\phi_0).
\]
\end{corollary}

\begin{proof}
Note that at a finite level $n$ (and for $d$ large enough so that the equalities in all the previous lemmas hold), the denominator in \eqref{eq:defgekeler} is 
\[
\frac{\#G(\integ_\ell/\ell^n)}{\#\A_G(\integ_\ell/\ell^n)}=
\frac{\#G^\der(\integ_\ell/\ell^n)\#\gp_m(\Z/\ell^n)}{\ell^{(\rank(G)-1)n}\#\gp_m(\Z/\ell^n)}
=\frac{\#G^\der(\integ_\ell/\ell^n)}{\ell^{(\rank(G)-1)n}}
= \frac{\ell^{(\dim(G)-1)(n-1)}\#G^\der(\integ_\ell/\ell)}{\ell^{(\rank(G)-1)n}}.
\]
By Lemma \ref{lem:U_n}(c), we have 
$\vol_{\abs{d\omega_G}}(\til
C_{d,n}(\gamma_0)) = \# C_{d,n}(\gamma_0)/\ell^{n\dim(G)}$, and 
by definition of the measure on the Steinberg quotient, 
$\vol_{\abs{d\omega_A}}(\til U_n(\gamma_0)) = \ell^{-n\rank(G)}$ (here we are using the fact that 
$\abs{\eta(\gamma_{X, \ell})}=1$ for $\ell\neq p$, so the absolute value of the $\G_m$-coordinate is $1$ on $\til U_n$).

Then for a given level $n$, we have
\begin{align*}
\frac{\# C_{(d,n)}(\gamma_0)}{\#G(\integ_\ell/\ell^n)/\#\A_G(\integ_\ell/\ell^n)} 
&=
\frac{\ell^{n\dim(G)}\vol_{\abs{d\omega_G}}(\til C_{(d,n)}(\gamma_0))
  \ell^{(\rank(G)-1)n}}{\ell^{(\dim(G)-1)(n-1)}\#G^\der(\integ_\ell/\ell)}
  \\
&=
  \frac{\ell^{\dim(G)-1}}{\#G^\der(\integ_\ell/\ell)}\frac{\vol_{\abs{d\omega_G}}(\til
  C_{(d,n)}(\gamma_0))}{\vol_{\abs{d\omega_A}}(\til U_n(\gamma_0))}.
\end{align*}
The result  now follows from Corollary \ref{lem:stein_oireprise}.
\end{proof}

\subsection{Calculation at $p$}
\label{sub:lisp}

Recall that we have fixed a maximal split torus $T_\spl \subset G$.
For any cocharacter $\lambda \in X_*(T_\spl)$ (and any power $q = p^e$ of
$p$), let $\psi_\lambda = \psi_{\lambda,q}$ be the characteristic
function of the double coset 
\[
D_{\lambda,q} = G(\integ_q) \lambda(p) G(\integ_q).
\]
By the Cartan decomposition, the collection of all $\psi_\lambda$ is a
basis for $\calh_G = \calh_{G,\rat_q}$, the Hecke algebra of functions
on $G(\rat_q)$ which are bi-$G(\integ_q)$-invariant.

Let $\mu_0$ be the cocharacter $p \mapsto \diag(p,\cdots, p, 1, \cdots, 1)$; it is the cocharacter associated to the Shimura variety
$\mathcal A_g$. 
 Define
\begin{align*}
\psi_{q,p} &= \psi_{\mu_0,q} = \one_{G(\integ_q)\diag(p, \cdots, p,
             1, \cdots, 1) G(\integ_q)} \\
\phi_{q,p} &= \psi_{e \mu_0, p} = \one_{G(\integ_p)\diag(q, \cdots, q,
             1, \cdots, 1) G(\integ_p)}.
\end{align*}
Recall that $\delta_0 =
\delta_{X/\ff_q}$ represents the absolute Frobenius of $X$, and that
$\gamma_0 := \Norm \delta_0$ lies in $G(\rat_p) \cap M(\integ_p)$ \eqref{E:normdelta}.

\begin{lemma}
\label{lem:ordcartan}
Let $[X,\lambda]/\ff_q$ be a principally
polarized abelian variety.  Suppose that either $X$ is ordinary or
that $q=p$ (and thus $e=1$).  Then
\[
\orbit{G(\rat_p)}{\gamma_0}\cap M(\integ_p) \subseteq D_{e\mu_0,p}.
\]
\end{lemma}

\begin{proof}
We identify $X_*(T_\spl)$ with 
\begin{equation}
\label{eq:defcochar}
\st{\ubar a = (a_1,\cdots,a_{2g}) \in \mathbb{Z}^{2g} : a_i+a_{g+i} = a_j+a_{g+j}\text{
    for } 1 \le i,j \le g
}.
\end{equation}
Suppose $\gamma \in D_{\ubar a, p} \subseteq G(\rat_p)$.  Then
$\ord_p(\eta(\gamma))$ is the common value of $a_i+a_{g+i}$; and
$\gamma$ stabilizes $V\tensor \integ_p$ -- that is, $\gamma \in
M(\integ_p)$ -- if and only if each $a_i \ge 0$.

Let $f(\ubar a) = \#\st{i:a_i=0}$.  If $\alpha \in D_{\ubar a}
\cap M(\integ_p)$,
then $f(\underline{a})$ is the rank of $\pi_1(\alpha)$ as an endomorphism of $V/pV$.

With these preparations, suppose $\gamma \in
\orbit{G(\rat_p)}{\gamma_0}\cap M(\integ_p)$.  Note that we have
$a_i+a_{g+i} = e$.

First, suppose $X$ is ordinary.  Then  exactly $g$ eigenvalues of $\gamma$ are
$p$-adic units.  Consequently, if $\gamma \in D_{\ubar a}$, then
$f(\ubar a) = g$. The only $\ubar a$  as in \eqref{eq:defcochar}
compatible with the symmetry and integrality constraints is $(e, \cdots, e, 0, \cdots,
0)$. 

Second, suppose $X$ has arbitrary Newton polygon but that $e=1$.
Again, the only $\ubar a$ such that $a_i+a_{g+i} = e = 1$ and each
$a_i \ge 0$ is $(1, \cdots, 1, 0, \cdots, 0)$.
\end{proof}

\begin{lemma}
\label{lem:fl}
Suppose that $[X,\lambda]/\ff_q$ is an ordinary, simple, principally
polarized abelian variety.  Then
\[
TO_{\delta_0}(\psi_{q,p}) = O_{\gamma_0}(\phi_{q,p}).
\]
\end{lemma}

\begin{proof}
There is a base change map $b = b_{G,\rat_q/\rat_p}: \calh_{G,\rat_q}
\to \calh_{G,\rat_p}$.  The fundamental lemma asserts that, if $\psi
\in \calh_{G,\rat_q}$, then stable twisted orbital integrals for
$\psi$ match with stable orbital integrals for $b\psi$.  For our
$\delta_0$ and $\gamma_0$, the adjective \emph{stable} is redundant
(Lemma \ref{lem:singleclass}), the case of the fundamental lemma we
need is \cite[Thm.\ 1.1]{clozel:fl}, and we have
\begin{equation}
\label{eq:fl}
TO_{\delta_0}(\psi_{q,p}) = O_{\gamma_0}(b \psi_{q,p}).
\end{equation}
While we will stop short of computing $b\psi_{q,p}$, we will find a
function which agrees with it on the orbit
$\orbit{G(\rat_p)}{\gamma_0}$.

The Satake transformation is an algebra homomorphism $\satake: \calh_{G,\rat_q}
\to \calh_{T_\spl,\rat_q}$ which maps $\calh_{G,\rat_q}$
isomorphically onto the subring $\calh_{T_\spl,\rat_q}^W$ of
invariants under the Weyl group.  It is compatible with base change, in the
sense that there is a commutative diagram
\[
\xymatrix{
\calh_{G,\rat_q} \ar[r]^-{\satake} \ar[d]^b & \calh_{T_\spl,\rat_q}
\iso \cx[X_*(T_\spl)] \ar[d]^b \\
\calh_{G,\rat_p} \ar[r]^{\satake} & \calh_{T_\spl,\rat_p}
}
\]
We exploit the following data about the Satake transform and the base
change map.

Under the canonical identification of $X_*(T_\spl)$ and
$X^*(\hat T_\spl)$, the character group of the dual torus, $\lambda
\in X_*(T_\spl)$ gives rise to a character of $\hat T_\spl$, and
thus a representation $V_\lambda$ of $\hat G$; let $\chi_\lambda$ be
its trace.  
We have
\[
\satake(\psi_{\mu,q}) = q^{\ang{\mu,\rho}}\chi_\mu +
\sum_{\lambda<\mu} a(\mu,\lambda)\chi_\lambda
\]
for certain numbers $a(\mu,\lambda)$, where as usual $\rho$ is the
half-sum of positive roots
\cite[(3.9)]{gross:satake}.

On one hand, following Gross \cite[(3.15)]{gross:satake} and Kottwitz \cite[(2.1.3)]{kottwitz:shimura},
   we observe that the weight $\mu_0=(1,\dots 1, 0,\dots, 0)$  is miniscule, and therefore 
\[
\satake(\psi_{\mu_0,q}) = q^{\ang{\mu_0,\rho}}\chi_{\mu_0}. 
\]
If we think of elements of $\C[X_\ast(T_\spl)]^W$ as polynomials in $2g$
variables $z_1, \dots, z_{2g}$, then (essentially by definition of the
highest weight and the fact that the multiplicity of the highest
weight in an irreducible representation is $1$ -- in our case the
representation in question is in fact the oscillator representation \cite[(3.15)]{gross:satake}) we find that
the leading term of 
$\satake(\psi_{\mu_0,q})$ is  $q^{\ang{\mu_0,\rho}}z_1\dots z_g$. 
By definition, the base change map takes $f\in \C[z_1, \dots, z_{2g}, z_1^{-1}, \dots, z_{2g}^{-1}]^W$ to $f(z_1^e, \dots, z_{2g}^e)$. 
Then 
\[
b(\satake(\psi_{q,p}))= q^{\ang{\mu_0,\rho}} z_1^e\dots z_g ^e+
\sum_{\lambda < e\mu_0} a(e\mu_0,\lambda)\chi_{\lambda}.
\]
On the other hand, we have
\begin{align*}
\satake(\phi_{q,p}) &=  p^{\ang{e\mu_0,\rho}} \chi_{e\mu_0} +
                           \sum_{\lambda < e\mu_0} b(e\mu_0,\lambda)
                           \chi_\lambda \\
&= q^{\ang{\mu_0,\rho}} z_1^e\dots z_g ^e +
                           \sum_{\lambda < e\mu_0} c(e\mu_0,\lambda)
                           \chi_\lambda.
\end{align*}
In these formulas, $a(e\mu_0,\lambda)$, $b(e\mu_0,\lambda)$ and
$c(e\mu_0,\lambda)$ are coefficients of lower weight monomials
that are ultimately irrelevant to our calculation.  In particular,
\[
\phi_{q,p} - \satake\inv(b(\satake(\psi_{q,p})))
\]
vanishes on $D_{e\mu_0,p} = G(\integ_p) e\mu_0(p) G(\integ_p)$.  

The last point to note is that the intersection of the support of this
 difference $\phi_{q,p} - \satake\inv(b(\satake(\psi_{q,p})))$ with
 the orbit of $\gamma_0$ is contained in $M(\Z_p)$.  Once we have
 shown this, the desired result 
follows from the fundamental lemma \eqref{eq:fl} combined with Lemma \ref{lem:ordcartan}. 
We start by observing that since the multiplier is a multiplicative map, it is constant on double $G(\integ_p)$-cosets. 
Therefore, for any double coset $D_{\ubar a,p}$ such that $D_{\ubar a,p}\cap {}^{G(\Q_p)}\gamma_0\neq\emptyset$, we have $a_i+a_{g+i}=e$. 
Now suppose $\lambda\leftrightarrow \ubar a$ is a dominant weight
satisfying this condition and further satisfying  $\lambda\le e\mu_0$.
Then we have $a_1\ge a_2\ge \cdots\ge a_g$ and $a_g \ge 0$ because $\lambda$ is dominant; and on the other hand,  $e-a_1\ge e-a_2 \ge \dots \ge e-a_g$, and $e-a_g \ge 0$ because of the
condition $\lambda \le e \mu_0$.  
Therefore in particular,   $a_{g+1}, \dots ,a_{2g}$ are non-negative, and thus $D_{\lambda,p} \subset M(\Z_p)$ (and in fact, we have also shown that $a_1=\dots =a_g$).  
\end{proof}

\begin{lemma} \label{lem:oi_at_p}
Suppose that either $X$ is ordinary or that $q=p$.
Then there exists $d(\gamma_0)$ such that 
$$O^\steinberg_{\gamma_0}(\phi_{q,p})=\lim_{n\to
  \infty}\frac{\vol_{\abs{d\omega_G}}(\til
  C_{d(\gamma_0),n}(\gamma_0))}{\vol_{\abs{d\omega_A}}(\til U_n(\gamma_0))}.$$
\end{lemma}
\begin{proof}
Suppose that $X$ is ordinary (but $q$ is an arbitrary power of $p$).
By Lemma \ref{lem:singleclass}, $\fc^{-1}(\fc(\gamma_0))$ is a single
$G(\Q_p)$-conjugacy class; the same  argument shows this is true for
elements in a small neighbourhood of 
$\fc(\gamma_0)$.
Thus, using (\ref{eq:st.vol}), 
$O^\steinberg_{\gamma_0}(\phi_{q,p})$ equals
$\lim_{n\to
  \infty}\frac{\vol_{\abs{d\omega_G}}\left(\fc^{-1}(\til{U_n}(\gamma_0))\cap
    D_{e\mu_0,p}\right)}{\vol_{\abs{d\omega_A}}(\til U_n(\gamma_0))}$. 
By Lemma \ref{lem:ordcartan}, we have 
$$\fc^{-1}(\til{U_n}(\gamma_0))\cap D_{e\mu_0,p} = 
\fc^{-1}(\til{U_n}(\gamma_0))\cap M(\Z_p).$$
 
Therefore, all we need to show is that 
for large enough $d$ and $n$, we have 
\begin{equation}\label{eq:gek_p}
\fc^{-1}(\til{U_n}(\gamma_0)) \cap M(\Z_p) =
 \til C_{d,n}(\gamma_0);
 \end{equation}
but this is essentially Corollary \ref{cor:McapturesG}(b). 

The case where $q=p$ follows from Lemma \ref{lem:stein_oireprise} and
the second case of Lemma
\ref{lem:ordcartan}.
\end{proof}

\begin{lemma}
\label{lem:SO-geom}
On the double coset $D_{e\mu_0,p}$ we have
\begin{equation}\label{eq:SO-geom}
\abs{d\omega_G} = q^{\frac{g(g+1)}2+1} \mu^\serre.
\end{equation} 
\end{lemma}
\begin{proof}
Let $K=G(\Z_p)$. First, observe that the measure $\mu^\serre$ on $G(\Q_p)\cap M(\Z_p)$ is both left- and right- $K$-invariant (since multiplication by an element of $G(\Z_p)$ yields a bijection on $\bmod p^n$-points).
Consider the decomposition of $D_{e\mu_0,p}$ into, say, left $K$-cosets: 
$D_{e\mu_0,p}=\sqcup_{i=1}^s g_i K$ (the number $s$ of these cosets was 
computed by Iwahori and Matsumoto but is not needed here). 
It follows from left $K$-invariance of $\mu^\serre$ that $\mu^\serre(g_iK)$ is the same for all $i$.   

Second, the measure $\abs{d\omega_G}$ is normalized so that each
$K$-coset has volume $\#G(\F_p)$. Thus, in order to compare the
measures $\mu^\serre$ and $\abs{d\omega_G}$, we need to compare the
cardinality $\#\pi_n(g_i K)$ of the reduction $\mod p^n$ of any such
coset $g_i K$ that is contained in 
$D_{e\mu_0,p}$ with $\#G(\F_p)$,
for sufficiently large $n$. (Note that $n=1$ is insufficient, 
because for all such cosets the reduction $\bmod p$ of any matrix in
$gK$ would be of lower rank. One  needs to go to $n>e$ for
the ratios $\frac{\#\pi_n(gK)}{p^{n\dim(G)}}$ to stabilize.)  
Since the answer does not depend on $g_i$, we can take
$g_0=e\mu_0(p)=\diag(q, \dots, q, 1\dots, 1)$.  
In other words, we need to compute the cardinality of the fibre of the map 
\begin{equation*}
  \xymatrix@R-2pc{
    \varphi_q: G(\Z/p^n) \ar[r]& M(\Z/p^n)\\
{\left[\begin{smallmatrix} A &  B \\ C & D\end{smallmatrix}\right]} \ar@{|->}[r]&
{\left[\begin{smallmatrix} qA &  qB \\ C & D\end{smallmatrix}\right].}
  }
\end{equation*}

For simplicity, we would like to move the calculation to the Lie algebra. 
Let $n\gg e$. Observe that if $\varphi_q(\gamma_1)=\varphi_q(\gamma_2)$ for $\gamma_1, \gamma_2\in G(\Z/p^n)$, then 
$\left[\begin{smallmatrix} qI_g&  0 \\ 0 & I_g\end{smallmatrix}\right](\gamma_1 \gamma_2^{-1} -I)=0$, 
where $I_g$ is the $g\times g$-identity matrix, and $I$ is the identity matrix in $M_{2g}$. 
This implies, in particular, that  $\gamma_1\gamma_2^{-1} \equiv I \bmod p^{n-e}$.
Then we can write the truncated exponential approximation: $\gamma_1\gamma_2^{-1}=I+X+\frac12X^2+\dots$ for 
some $X\in \fg(\Z_p)$; in particular,  there exists  $X\in\fg(\Z_p)$ such that $\gamma_1\gamma_2^{-1}\equiv I+X \mod p^{2(n-e)}$, and thus 
the kernel of the map $\varphi_q$ is in bijection with the set of 
$(X\bmod p^n)$ for $X\in \fg(\Z_p)$ such that 
$\left[\begin{smallmatrix} qI_g&  0 \\ 0 & I_g\end{smallmatrix}\right] X \equiv 0 \mod p^n$. 

We have $\fg =\mathfrak{sp}_{2g}\oplus \mathfrak{z}$, where $\mathfrak{z}$ is the $1$-dimensional Lie algebra of the centre. It will be convenient to decompose it further: let $\mathfrak h$ be the Cartan subalgebra of $\mathfrak{sp}_{2g}$ consisting of diagonal matrices, and let $V$ consist of matrices whose diagonal entries are all zero; then
$$\fg=(\mathfrak{z}\oplus \mathfrak h) \oplus V.$$
Consider the action of multiplication by $\left[\begin{smallmatrix} qI_g&  0 \\ 0 & I_g\end{smallmatrix}\right]$ on each term of this direct sum decomposition. 

On the term $\mathfrak{z}\oplus \mathfrak h$ it acts by 
$\diag(a_1, \dots, a_{2g})\mapsto \diag (q a_1, \dots, q a_g, a_{g+1}, \dots, a_{2g})$, which in the 
$\mathfrak{z}\oplus \mathfrak h$-coordinates can be written as (recalling that $a_i+a_{g+i}=z$ is independent of $i$): 
$$\begin{aligned}
&\frac{z}2\oplus \left(\frac{z}2-a_{g+1}, \dots \frac{z}2-a_{2g}, -\frac{z}2+a_{g+1}, \dots -\frac{z}2+a_{2g}\right) \\ &\mapsto \frac{qz}2\oplus
\left(\frac{qz}2-\frac{(q+1)a_{g+1}}2, \dots \frac{qz}2-\frac{(q+1)a_{2g}}2, -\frac{qz}2+\frac{(q+1)a_{g+1}}2, \dots -\frac{qz}2+\frac{(q+1)a_{2g}}2\right).
\end{aligned}
$$
The only points $(z, a_{g+1}, \dots, a_{2g})$  that are killed $(\bmod p^n)$ by this map are of the form 
$(z', 0, \dots, 0)$ with $qz'=0$; so there are $q$ of them.   

Next consider an element $X=\left[\begin{smallmatrix} A &  B \\ C & D\end{smallmatrix}\right] \in V$. Then $A$ is determined by $D$, and $B$ is skew-symmetric (up to a permutation of rows and columns).  Multiplication by  $\left[\begin{smallmatrix} qI_g&  0 \\ 0 & I_g\end{smallmatrix}\right]$
scales each entry of  $A$ and $B$ by a factor of $q$, and does not change $C$ and $D$.
Since $A$ is determined by $D$, the elements $X$ killed by this map are in bijection with symmetric matrices $B$ 
with entries in $\Z/p^n$ that are killed by multiplication by $q$. 
Since the space of such matrices is a $g(g+1)/2$-dimensional 
linear space, the number of such matrices $B$ is $q^{g(g+1)/2}$.

Thus, we have computed that $\abs{d\omega_G} = q^{\frac{g(g+1)}2+1} \mu^\serre$ on the double coset $D_{e\mu_0,p}$. 
Combining this with (\ref{eq:serre_at_p}), we get: 
$$\nu_p([X, \lambda])= 
\frac{q p^{\dim(G^\der)}}{\#G^\der(\integ_p/p)}
\frac{q^{-\frac{g(g+1)}2-1}\vol_{\abs{d\omega_G}}(\til
  C_{d,n}(\gamma_0))}{\vol_{\abs{d\omega_A}}(\til U_n(\gamma_0))} 
  = q^{-\frac{g(g+1)}2}\frac{p^{\dim(G^\der)}}{\#G^\der(\integ_p/p)} O^\steinberg_{\gamma_0}(\phi_{q,p}),
$$
which completes the proof.
\end{proof}

\begin{corollary}\label{cor:gek2stein_atp}
Suppose that either $X$ is ordinary or that $q = p$.
For $\ell =p$,
the Gekeler ratio \eqref{eq:defgekeler} is related to the geometric orbital integral by 
\[
\nu_p([X,\lambda]) = q^{-\frac{g(g+1)}2}\frac{p^{\dim(G^\der)}}{\#G^\der(\integ_p/p)}O^\steinberg_{\gamma_0}(\phi_{q,p}).
\]
\end{corollary}
\begin{proof} 
First observe that $\vol_{\abs{d\omega_A}}(\til U_n(\gamma_0))=q p^{-n\rank(G)}$, since  we are using the invariant measure on the $\G_m$-factor of $\A_G=\A^{\rank(G)-1}\times \G_m$, and for $\gamma_0$ (and therefore, for all points in $\til U_n$), that coordinate is the multiplier, with absolute value $q^{-1}$.  
Thus,  by Lemma \ref{lem:U_n} (c) and the same argument as in Corollary \ref{cor:gek2stein}, 
we have that for  $d>d(\gamma_0)$, 
\begin{equation}\label{eq:serre_at_p}
\nu_p([X, \lambda])=\lim_{n\to\infty} \frac{\#C_{d,n}(\gamma_0)}{\#G(\Z/p^n\Z)/\#\A_G(\Z/p^n\Z)} = 
\frac{q p^{\dim(G)-1}}{\#G^\der(\Z/p\Z)}\lim_{n\to
  \infty}
  \frac{ \vol_{\mu^\serre}(\til
  C_{d,n}(\gamma_0))}{\vol_{\abs{d\omega_A}}(\til U_n(\gamma_0))}. 
  \end{equation}
   The ratio inside the limit  on the right-hand side  is the same as the ratio  in Lemma \ref{lem:oi_at_p}, except that the measure in the numerator is the Serre-Oesterl\'e measure $\mu^\serre$ rather than the measure $\abs{d\omega_G}$. (Both measures are defined on $G(\Q_p)\cap M(\Z_p)$.)
Thus, to prove the corollary, we just need to compute  the conversion factor between the restrictions of the measures $\mu^\serre$ and  ${\abs{d\omega_G}}$  to the support of $\phi_{q,p}$, which is the content of Lemma \ref{lem:SO-geom}.
\end{proof}

\section{The product formula}\label{sec:global} 

Now that the relationship between the ratios $\nu_\ell$ and orbital integrals (with respect to the geometric measure) is established, 
we can translate the formula of Langlands and Kottwitz \eqref{eqlk} into a Siegel-style product formula for the ratios, thus obtaining our main theorem. 
Recall the notation of \S \ref{secdef}, in particular,  the element $\gamma_{[X, \lambda]}\in G(\A_f)$ associated with the isogeny class of 
$[X, \lambda]$, and its centralizer $T=T_{[X, \lambda]}$. Here in order to ease the notation we drop all the subscripts $[X, \lambda]$. 
Note that there is some flexibility in the choice of the measures in the  Langlands and Kottwitz formula, but the measures need to be normalized by normalizing the measures on $G(\Q_\ell)$ and on $T(\Q_\ell)$ separately.
We will use the canonical measure $d\mu_G^\can$ on $G(\Q_\ell)$ for every prime $\ell$, and the Tamagawa  
measure $\mu_T^\tama$ (defined in detail below) on $T(\Q_\ell)$ for all $\ell$. This gives a convergent product measure globally, since the local  orbital integrals equal $1$ at almost all places with respect to this measure. 
Since Gekeler-style ratios are expressed in terms of the geometric measure on orbits, we need to calculate the conversion factor between the geometric measure and  the quotient $\mu_G^\can/\mu_T^\tama$.  
We start with a quick review of the definition of $\mu_T^\tama$ in order to introduce all the relevant notation. 

\subsection{Tamagawa measure}\label{sub:tama} 
Let $S$ be an algebraic torus; here we only discuss the setting where $S$ is defined over $\Q$. 
The character group of $S$ is the free $\integ$-module $X^\ast(S) =
\hat S$ in Ono's notation (we emphasize that this is the lattice of the characters defined over $\bar \rat$). 
If $F$ is any field containing $\rat$, we let $(\hat S)_F$ be
the subgroup of characters of $S$ which are defined over $F$. 

As usual, we have $S(\aff)$ and $S(\aff_f)$, the points of $S$ with
values in, respectively, the ring of adeles and the ring of finite
adeles.  The (finite) adeles come equipped with the product absolute value
$\abs{\cdot}_\aff$, and we set
\begin{align*}
S(\aff)^1 &= \st{ s \in S(\aff) : \forall \chi \in (\hat S)_\rat,
  \abs{\chi(s)}_\aff = 1}.
\end{align*}



Let $F$ be a Galois extension which splits $S$.
Then the character lattice $X^\ast(S)$ can be viewed as a
$\Gal(F/\Q)$-module, and this module  uniquely determines $S$ up to
isomorphism.   
We denote this representation by $\sigma_S$, and let $L(s, \sigma_S)=\prod_\ell L_\ell(s, \sigma_S)$ be the corresponding Artin $L$-function
(see \cite{bitan} for a modern treatment).
Let $r$ be the multiplicity of the trivial representation in $\sigma_S$. 
By definition, 
$$\rho_S:=\lim_{s\to 1}(s-1)^r L(s, \sigma_S).$$

Let $\omega$ be an invariant
  gauge form on $S$. (In particular, $\omega$ is defined over $\rat$.)
 Set
\[\omega^\tama=\omega_{\infty}\prod_\ell L_\ell (1,\sigma_{S})\omega_\ell,\]
 where $\omega_\ell$ is the invariant volume form on $S(\mathbb{Q}_\ell)$ induced by $\omega$.

By the product formula, as long as $\omega$ is defined over $\Q$, none
of the global invariants depend on the normalization of $\omega$. 

 Let $\chi_1, \cdots, \chi_r$ be a basis
  for $(\hat S)_\rat$, and define a  map  $\Lambda$ by
\[\xymatrix@R-2pc{
S(\A)\ar[r]^\Lambda& (\R_+^{\times})^{r}\\
 x \ar@{|->}[r]&
(\abs{\chi_1(x)}_{\A},\cdots,\abs{\chi_r(x)}_{\A}).
}\]
(In the cases  of interest, when $S= T^\der$ or $S=T$, we have $r=0$ or $r=1$, respectively.)
Then  $\Lambda$ induces an isomorphism 
\[\xymatrix{\tilde{\Lambda}:S(\A)/S(\A)^1 \ar[r]^{\quad \quad\sim} & (\R_+^{\times})^r}.\]

(Of course, both sides are trivial if $S$ is anisotropic.)

Define  $d\tilde{t}$ by
\[d\tilde{t}:=\tilde{\Lambda}^*(\prod_{k=1}^r\frac{dt}{t}).\]
Let $dS_{\Q}$ be the counting measure on $S(\Q)$.
The Tamagawa measure on $S(\A)^1$ defined by Ono \cite[(3.5.2)]{ono:arithmetic_tori}, (taking into account that  in our case the base field denoted by $k$ in \cite{ono:arithmetic_tori} is $\Q$) is the measure $\mu^\tama$ that makes the following equality true:
\begin{equation}\label{eq:tama}
 \rho_S^{-1}\omega^\tama = d\tilde{t}\,\mu^\tama dS_\Q.
\end{equation} 
 The Tamagawa number is defined
  by
\[\tau_S=\int_{S^1(\A)/S(\Q)}\mu^\tama.\]
We will also make use of the differential form on $S$ that we denote by $\omega_S$ 
(this notation agrees with that of \cite{langlands-frenkel-ngo}). 
We define: 
\begin{equation}\label{eq:omega_T}
\omega_S:=\frac{d\chi_1}{\chi_1}\wedge \dots \wedge \frac{d\chi_d}{\chi_d},
\end{equation}
where $d$ is the rank of $X^\ast(S)$. This form is, a priori, defined over $\bar \Q$. 
However, in fact there exists $D\in \Q$  such that $\omega_S/\sqrt{D}$ is defined over $\Q$.  
(see \cite[Corollary 3.7]{gan-gross:haar}). 
Since $\sqrt{|D|}|\prod_{\ell}\sqrt{|D|_\ell} =1$,  in fact we can use the form 
$\omega_S$ instead of $\omega$ in the definition of the Tamagawa measure, even though
it is not quite defined over $\Q$. 
Specifically, we will from now on work with the form 
\begin{equation}\label{eq:omega-S-tama}
\omega_S^\tama=(\omega_S)_{\infty}\prod_\ell L_\ell(1,\sigma_{S})(\omega_S)_\ell \quad \text{on} \quad S(\A).
\end{equation}
We denote the product over the finite primes by $\omega_{S,f}$, i.e., write 
$\omega_S^\tama=(\omega_S)_{\infty}\omega_{S,f}$; the form $\omega_{S,f}$ defines a measure on $S(\A_f)$, the set of points of $S$ over the finite adeles.

\subsection{The measure $\mu^\tama$ vs. geometric measure}\label{subsec:comparison}
This section is based on \cite{langlands-frenkel-ngo}. 
We recall that the measure on orbits that we call $\mu^{\geom}$ is constructed as a quotient of the measure $|\omega_G|$ by the measure $|\omega_A|$ on our 
space $\A_G$, which is a `na\"ive version' of the Steinberg-Hitchin base 
(see \cite[\S 3.7]{jg:singapore} for a detailed comparison of the space $\A_G$, and the measure on it, with the actual Steinberg-Hitchin base that is used in \cite{langlands-frenkel-ngo}). 

Consider the measure on $T$ defined by the form $\omega_T^\tama$ at every place (its only difference from the Tamagawa measure on $T$ is in the 
global factor $\rho_T$). 
For every finite prime $\ell$, let $\mu_{\gamma_\ell}^\tama$ be the measure on the orbit of $\gamma_\ell$ in $G(\Q_\ell)$ 
obtained as the quotient of the measure $\omega^\can$ on $G$ that gives the maximal compact subgroup $G(\Z_\ell)$ volume $1$, by the measure 
$|\omega_T^\tama|_\ell$ on $T$. 

The following proposition is an adaptation of the equality (3.31) of \cite{langlands-frenkel-ngo} to our setting. 

\begin{proposition}(\cite[Proposition 3.29; (3.31)]{langlands-frenkel-ngo})\label{prop:ratio}
We have 
\[\mu_{\gamma,\ell}^\geom=\abs{\eta(\gamma)}_\ell^{-\frac{g(g+1)}{4}}\sqrt{\abs{D(\gamma)}_\ell}{\vol_{\omega_G}\left(G(\Z_\ell)\right)}
L_\ell(1, \sigma_T){\mu}_{\ell}^\tama,\]
where $\eta(\gamma)$ is the multiplier of $\gamma$. 
\end{proposition}

\begin{proof} For $\gamma\in G^\der$,  this is equivalent to the relation  (3.31) of
 \cite{langlands-frenkel-ngo};
 the additional factor  ${\vol_{\omega_G}\left(G(\Z_\ell)\right)}$ on the right appears here because 
in (3.31) of \cite{langlands-frenkel-ngo}, the same measure on $G$ needs to be used on both sides of the equation; here we are using the measure 
$|\omega_G|$ on the left, and the measure $|\omega_G^\can| = |\omega_G|/\vol(G(\Z_\ell))$ on the right; so this correction factor is needed. 
More precisely, the relation (3.30) in \cite{langlands-frenkel-ngo} (which we also reproved  in \S 4.2.1 of \cite{achtergordon17})
asserts that for a semisimple group $G$, 
\[\mu_{\gamma,\ell}^\geom=\sqrt{\abs{D(\gamma)}_\ell}
|\omega_{T\backslash G}|,\] where 
$\omega_{T\backslash G}$ is the quotient of the measure $\omega_G$ by the measure $\omega_T$ on $T$ defined above in \S\ref{sub:tama}, which is the same as the measure $\omega_T$ in \cite{langlands-frenkel-ngo}. 
Since by definition (and the remark at the end of \S \ref{sub:tama} which allows us to use the form $\omega_T$ in the definition of the Tamagawa measure), \\
$\omega_{T, \ell}^\tama = L_\ell(1, \sigma_T)\omega_{T, \ell}$, this  proves the Proposition for $\gamma\in G^\der$.  
For general $\gamma$, the factor $\abs{\eta(\gamma)}^{-\frac{g(g+1)}{4}}$ appears on the right-hand side because we are using the space $\A_G$ instead of the 
Steinberg-Hitchin base of $\gsp_{2g}$. This factor is calculated 
by considering the action of the centre of $G$ on all the measure spaces involved. This is explained in detail in \S3.7 of \cite{jg:singapore}. 
\end{proof}

\subsection{Proof of Theorem \ref{thmain}}
For convenience, we list here the results we proved above about the ratios $\nu_\ell$ and the relevant orbital integrals: 
\begin{enumerate}
     \item (Corollary 4.7) at $\ell\neq p$
     \[
\nu_\ell([X,\lambda]) = \frac{\ell^{\dim(G^\der)}}{\#G^\der(\integ_\ell/\ell)}O^{geom}_{\gamma_0}(\phi_0).
\]
     \item (Corollary 4.12) at $p$
     \[
\nu_p([X,\lambda]) = q^{-\frac{g(g+1)}2}\frac{p^{\dim(G^\der)}}{\#G^\der(\integ_p/p)}O^{geom}_{\gamma_0}(\phi_{q,p}).
\]
  \end{enumerate}

Combining these with Proposition \ref{prop:ratio}, and 
observing that for $\ell\neq p$, we have $\abs{\eta(\gamma)}_\ell=\abs{\det(\gamma)}_\ell=1$, and at $p$, we have 
$\abs{\eta(\gamma)}_p=q^{-1}$,  we obtain: 
\begin{equation}\label{eq: p_final}
\begin{aligned}
\nu_p([X, \lambda])& = q^{-\frac{g(g+1)}{2}}\frac{p^{\dim(G^\der)}}{\#G^\der(\F_p)} O_{\gamma_0}^\geom(\phi_{q,p})
\\
&=
q^{-\frac{g(g+1)}2} q^{\frac{g(g+1)}4}\sqrt{\abs{D(\gamma)}_p}
{\vol_{\omega_G}\left(G(\Z_p)\right)}L_p(1, \sigma_T)
\frac{p^{\dim(G^\der)}}{\#G^\der(\F_p)}O_{\gamma_0}^\tama(\phi_{q,p})\\
&= 
q^{-\frac{g(g+1)}4}{\sqrt{\abs{D(\gamma)}_p}}L_p(1, \sigma_{T/G})O_{\gamma_0}^\tama(\phi_{q,p}), \quad \text{and} \\
\nu_\ell([X, \lambda])& = \frac{\ell^{\dim(G^\der)}}{\#G^\der(\F_\ell)} O_{\gamma_0}^\geom(\phi_{\ell})\\
&=
\sqrt{\abs{D(\gamma)}_\ell}
{\vol_{\omega_G}\left(G(\Z_\ell)\right)}L_\ell(1, \sigma_T)
\frac{\ell^{\dim(G^\der)}}{\#G^\der(\F_\ell)}O_{\gamma_0}^\tama(\phi_\ell)\\
&= 
{\sqrt{\abs{D(\gamma)}_\ell}}L_\ell(1, \sigma_{T/G})O_{\gamma_0}^\tama(\phi_\ell).
\end{aligned}
\end{equation} 
Here the notation $L_\ell(1, \sigma_{T/G})$ stands for $L_\ell(1, \sigma_T)(1-\frac1{\ell})$ (including the case $\ell=p$), which agrees with the use of this notation in \cite{langlands-frenkel-ngo},  and the 
last equality in both cases follows from (\ref{eq:can_meas}).

%

Taking a product of these over all primes $\ell$,
 and recalling 
 the product formula for absolute values, we obtain: 
\begin{equation}\label{eq:prod1}
\prod_\ell \nu_\ell = q^{-\frac{g(g+1)}4}|D(\gamma)|^{-1/2}
L(1, \sigma_{T/G})
O^\tama_\gamma, 
\end{equation} 
where $O^\tama_\gamma$ stands for the product of orbital integrals in the Langlands-Kottwitz formula, with the measure on each factor given by 
$\omega_{T,\ell}^\tama$. 

\begin{lemma}
$$\vol_{\omega_{T,f}^\tama}(T(\Q)\backslash T(\A_f))= \rho_T \frac{1}{(2\pi)^g}\tau(T).$$
\end{lemma}
\begin{proof} We start by emphasizing that while in the definition of the Tamagawa number, one starts with an arbitrary differential form defined over $\Q$, since here we have split off the infinite places, the specific choice of the differential form matters. This choice 
(dictated by our calculations above, specifically, Proposition \ref{prop:ratio}) is the form $\omega_T$ of (\ref{eq:omega_T}) (which is not defined over $\Q$ but can still be used,
 as discussed at the end of \S \ref{sub:tama}). From this form, we make the form 
 $\omega_T^\tama=(\omega_T)_\infty\omega_{T,f}$ as in  (\ref{eq:omega-S-tama}). 
 We need to identify explicitly the component at the infinite place of the form $\omega_T^\tama$, which is the same as the 
 component $(\omega_T)_\infty$ of the differential form $\omega_T$. 
 We have a convenient basis for the character lattice  of $T$, cf. \S \ref{subsec:torus}. Define the coordinates $(z_1, \dots, z_g, \lambda)$ on 
  $T(\C)$ such that $T(\C)=\{z_1, \dots, z_g, \lambda z_1^{-1}, \dots, \lambda z_g^{-1}) \}$, and define the the characters 
 $\chi_i(z_1, \dots, z_g, \ z_1^{-1}, \dots, \lambda z_g^{-1})=z_i$, for $i=1, \dots, g$, and the multiplier 
 $\eta(z_1, \dots, z_g, \ z_1^{-1}, \dots, \lambda z_g^{-1}):=\lambda$. (Note that in the Appendix, the character lattice of $T$ is described as a quotient of $\Z^{2g}$ instead, which is convenient for the cohomology computations; we do not use this description here.)
We write every element of $T(\A)$ as $a=a_f a_\infty$, where $a_f$ has the infinity component $1$ and $a_\infty=(1,z_1, \dots, z_g, \lambda)$
has all the components at the finite places equal to $1$.
In this notation, $T(\A)^1$ is defined by the condition $|z_i|=|\lambda| = \|a_f\|^{-1/g}$. 
 We note that the character $\eta$ coincides with the map $\tilde \Lambda$ from $T(\A)/T(\A)^1$ to $\R_+$ defined in \S \ref{sub:tama}. 
 Then by the definition of $\mu^\tama$, it is the volume form on $T(\A)^1$ given by 
 $$\mu^\tama\wedge \frac{d{\eta}}{\eta} = 
 \frac{d\chi_1(a_f a_\infty)}{a_fa_\infty}\wedge\dots \wedge  \frac {d\chi_g(a_f a_\infty)}{a_fa_\infty}\wedge \frac{d{\eta}}{\eta}, 
 $$
 and thus its component at $\infty$ is the form $\mu^\tama_\infty = \frac{dz_1}{z_1}\wedge\dots \wedge  \frac {d z_g }{z_g}$ on 
 $T(\A)^1_{\infty} \simeq (S^1)^g$. 
 It is an easy exercise  (see \cite[(2.8)]{jg:singapore}) that the form $dz/z$ gives precisely the arc length measure on the unit circle. 
 Thus, we get: 
 $$\tau_T = \vol_{\mu^\tama}(T(\Q)\backslash T(\A)^1) = \rho_T^{-1}\vol_{\omega_T^\tama}(T(\Q)\backslash T(\A)^1) =
 \rho_T^{-1}\vol_{\omega_T^\tama}(T(\Q)\backslash T(\A_f))(2\pi)^g, $$
  which completes the proof. 
\end{proof} 

Now we can complete the proof of the theorem. 
By  the Langlands-Kottwitz formula (in which we choose the Tamagawa measure on $T$),
we have 
$$\widetilde{\#} I([X, \lambda])= \vol_{\omega_T^\tama}(T(\Q)\backslash T(\A_f)) O^\tama_\gamma = \rho_T \frac{1}{(2\pi)^g}\tau_T O^\tama_\gamma.$$
On the other hand, by (\ref{eq:prod1}), we have 
$$\prod_\ell \nu_\ell = q^{-\frac{g(g+1)}4}|D(\gamma)|^{-1/2}
L(1, \sigma_{T/G})
O^\tama_\gamma.$$
It remains to recall that we have defined $\nu_\infty = \frac{|D(\gamma)|^{1/2}}{(2\pi)^g}$, and note that 
the Euler product for $L(1, \sigma_{T/G})$ is conditionally convergent and equals $\rho_T$ (cf. \cite[(3.25)]{langlands-frenkel-ngo}), 
since $\rho_{\G_m}=1$, the residue of the Riemann zeta-function at $s=1$. 
Thus, 
$$ q^{-\frac{g(g+1)}4} \tau_T \nu_\infty \prod_\ell \nu_\ell  = \frac{\tau_T}{(2\pi)^g} L(1, \sigma_{T/G}) O^\tama_\gamma = 
\vol_{\omega_T^\tama}(T(\Q)\backslash T(\A_f)) O^\tama_\gamma,
$$
which completes the proof.

\section{Complements}
\label{S:examples}

For the convenience of a hypothetical reader interested in explicit
calculations, we collect here some reminders concerning the terms
which arise in \eqref{eqmain}.

\subsection{$\nu_\infty$}

Recall that we have defined \eqref{eq:gekinf} $\nu_\infty([X,\lambda])$
as $\sqrt{\abs{D(\gamma_0)}}/(2\pi)^g$, where $D(\gamma_0)$ is the
Weyl discriminant $D(\gamma_0) =
\prod_{\alpha\in\Phi}(1-\alpha(\gamma_0))$, the product being over all
roots of $G$.  We may relate this to the (polynomial) discriminant of
$f_{X/\ff_q}(T)$, the characteristic polynomial of Frobenius, as
follows.

\subsubsection{Weyl discriminants}
\label{subsub:weyldisc}

Explicitly, $\gamma_0$ has multiplier $\lambda_0 :=
\eta(\gamma_0) = q$.  Write the (complex) eigenvalues of (a
$\rat$-representative of) $\gamma_0$ -- equivalently, the roots of
$f_{X/\ff_q}(T)$ -- as $(\lambda_1, \cdots, \lambda_g,
\lambda_0/\lambda_1, \cdots, \lambda_0/\lambda_g)$.  Then
\begin{align*}
D(\gamma)
&= \prod_{1  \le i < j \le g} \delta_{ij} \cdot \prod_{1 \le i \le g}
\delta_{i} \\
\intertext{where}
\delta_{ij} &=
(1-\lambda_i/\lambda_j)(1-\lambda_j/\lambda_i)(1-\lambda_i\lambda_j/\lambda_0)(1-\lambda_0/(\lambda_i\lambda_j))
\\
\delta_i &= (1-\lambda_i^2/\lambda_0)(1-\lambda_0/\lambda_i^2).
\end{align*}
Possibly after reordering the conjugate pairs $\st{\lambda_i,
  \lambda_0/\lambda_i}$, we may and do assume that $\lambda_j = \sqrt
q \exp(i\theta_j)$ with $0 \le \theta_j < \pi$.  Then
\begin{align*}
\delta_{ij} &= (2\cos(\theta_i)-2\cos(\theta_j))^2\\
\delta_j &= 4\sin^2(\theta_j).
\end{align*}

\subsubsection{Elliptic curves}
\label{subsub:nuinfg1}

Suppose that $[X,\lambda]$ is an elliptic curve with its canonical
principal polarization, say with characteristic polynomial of
Frobenius $T^2-aT+q$. 
Then $a = 2\sqrt q \cos(\theta)$, and $D(\gamma_0) = 4\sin^2(\theta)=
4 - \frac{a^2}{q}$, and 
\[
\nu_\infty([X,\lambda]/\ff_q) = \oneover{2\pi}
\sqrt{\abs{D(\gamma_0)}} = \oneover{\pi}\sqrt{1-\frac{a^2}{4q}}.
\]
Note that this term is \emph{half} the archimedean term introduced in
\cite[(3.3)]{gekeler03} (when $q=p$) and
\cite[(2-7)]{achtergordon17}.  For purposes of comparison, we
summarize this relationship by writing
\[
\nu_\infty([X,\lambda]) = \half 1 \nu_\infty^{\mathrm{Gek}}([X,\lambda]) =
\half 1 \nu_\infty^{\mathrm{AG}}([X,\lambda]).
\]

\subsubsection{Polynomial discriminants}

To facilitate comparison with
\cite{achterwilliams15,gekeler03,gerhardwilliams19}, we express $D(\gamma_0)$ in
terms of polynomial discriminants.  Let $f(T) = f_{X/\ff_q}(T)$, and
let $f^+(T) = f^+_{X/\ff_q}(T)$ be the
minimal polynomial of the sum of $\gamma_0$ and its adjoint, so that 
\[
f^+_{X/\ff_q}(T) = \prod_{1 \le j \le g}(T-(\lambda_j+q/\lambda_j)).
\]
Note that $\rat[T]/f^+(T) \iso K^+$, the maximal totally
real subalgebra of the endomorphism algebra of $X$.

\begin{lemma}
We have
\[
\frac{\disc(f(T))}{\disc(f^+(T))} = (-1)^g q^{\half{g(3g-1)}}D(\gamma_0).
\]
\end{lemma}

\begin{proof}
On one hand, 
\begin{align*}
\disc(f(T)) &= \prod_{1 \le i < j \le g} \alpha_{ij}^2 \prod_{1 \le i
  \le g} \alpha_{i}^2
\intertext{where}
\alpha_{ij} &= (\lambda_i - \lambda_j)(\lambda_i -
\lambda_0/\lambda_j)(\lambda_0/\lambda_i-\lambda_j)(\lambda_0/\lambda_i-\lambda_0/\lambda_j)
\intertext{and }
\alpha_{i} &= (\lambda_i-\lambda_0/\lambda_i).
\end{align*}
On the other hand, 
\begin{align*}
\disc(f^+(T)) &= \prod_{1 \le i < j \le g} \beta_{ij}^2
\intertext{where}
\beta_{ij} &= (\lambda_i+\lambda_0/\lambda_i -
(\lambda_j+\lambda_0/\lambda_j)).
\end{align*}
Now use this to evaluate $\disc(f(T))/\disc(f^+(T))$, while bearing in
mind that
\[
\frac{\alpha_{ij}}{\beta_{ij}^2} = \lambda_0
\text{ and }\frac{\alpha_{ij}}{\delta_{ij}} = \lambda_0^2\text{ and
}\frac{\alpha_{i}^2}{\delta_{i}} = -\lambda_0.
\]
\end{proof}

\subsection{$\nu_\ell$}

Gekeler  \cite{gekeler03} observed that, for elliptic curves, his product formula essentially computes an
L-function; a similar phenomenon has been observed in other contexts,
as well \cite{achterwilliams15,gerhardwilliams19}.  We briefly explain how
this relates to \eqref{eqmain}.  This detour also has the modest
benefit of showing that the right-hand side of \eqref{eqmain}
converges, albeit conditionally.

\subsubsection{Zeta functions}

We express the zeta function of a number field $M$ as $\zeta_M(s) =
\prod_\ell \zeta_{M,\ell}(s)$, where $\zeta_{M,\ell}(s) = \prod_{\lambda|\ell} 
(1-\Norm_{M/\rat}(\lambda)^{-s})\inv$.  For a direct sum 
$M = \oplus_{i=1}^t M_i$ of such fields we write $\zeta_{M,\ell}(s) =
\prod_i \zeta_{M_i,\ell}(s)$; the product over all primes yields  $\zeta_M(s) = \prod_i \zeta_{M_i}(s)$. 

Recall  that to a torus $S/\rat$
one associates an Artin L-function $L(s,\sigma_S) = \prod_\ell
L_\ell(s,\sigma_S)$.  This construction is multiplicative for exact
sequences of tori, and for a finite direct sum  $M$ of number fields one has
$L(s,\sigma_{\res_{M/\rat}\gp_m})=\zeta_M(s)$. (It may be worth
recalling that $\res_{M/\rat}\gp_m \iso \oplus \res_{M_i/\rat}\gp_m$.)

If $\ell$ is unramified in some splitting field
for $S$, then (cf.  \cite[2.8]{bitan},  \cite[14.3]{voskresenskii})
one has
\[
\# S(\ff_\ell) = \ell^{\dim S} L_\ell(1,\sigma_S)\inv.
\]

\begin{lemma}
\label{lem:zetafactor}
Suppose that $\ell \nmid 2 p \disc(f_{X/\ff_q}(T))$.  Then
\[
\nu_\ell([X,\lambda]) = \frac{\zeta_{K,\ell}(1)}{\zeta_{K^+,\ell}(1)}.
\]
\end{lemma}

\begin{proof}
By Lemma \ref{lem:gekelerunram}
\begin{align*}
\nu_\ell([X,\lambda]) &= \frac{\#\st{\gamma \in G(\ff_\ell) : 
                        \gamma \twiddle \pi_1(\gamma_0)}}{\#
                        G(\ff_\ell)/\#\A_G(\ff_\ell)} \\
&= \frac{\#
  G(\ff_\ell)/\#T(\ff_\ell)}{\#G(\ff_\ell)/\left( \ell^{g}\#\gp_m(\ff_\ell)\right)} \\
&= \ell^g\frac{\#\gp_m(\ff_\ell)}{\#T(\ff_\ell)} 
= \frac{L_\ell(1,\sigma_T)}{L_\ell(1,\sigma_{\gp_m})},
\intertext{since $\dim T = g+1$.  Using \eqref{eq:diagT}, first to
     see that $L(s,\sigma_T) = L(s,\sigma_{T^\der})
     L(s,\sigma_{\gp_m})$ and second to compute $L(s,\sigma_{T^\der})$,
      we recognize this as
}
&=\frac{\zeta_{K,\ell}(1)}{\zeta_{K^+,\ell}(1)}.
\end{align*}
\end{proof}

Since $\zeta_K(s)$ and $\zeta_{K^+}(s)$ both have a simple pole at
$s=1$, we immediately deduce:

\begin{corollary}
The right-hand side of \eqref{eqmain} converges conditionally.
\end{corollary}

Moreover, up to a finite factor $B([X,\lambda])$, we can express $\wnum I([X,\lambda],\ff_q)$ in terms of familiar quantities:

\begin{corollary}
  We have
  \begin{align*}
    \wnum I([X,\lambda],\ff_q) &= \tau_T \frac{q^{-\frac{g(3g-1)}{4}}}{(2\pi)^g} \sqrt{\abs{\frac{\disc(f)}{\disc(f^+)}}} B([X,\lambda]) \lim_{s \to 1^+}\frac{\zeta_K(s)}{\zeta_{K^+}(s)}
    \intertext{where }
    B([X,\lambda]) &= \prod_{\ell | 2p\disc(f)} \frac{\zeta_{K^+,\ell}(1)}{\zeta_{K,\ell}(1)} \nu_\ell([X,\lambda]).
  \end{align*}
\end{corollary}

\subsubsection{Elliptic curves}

If $[X,\lambda]$ is an elliptic curve, let $\chi$ be the quadratic
character associated to the imaginary quadratic field $K$.  Then $K^+
= \rat$, and $\zeta_K(s)/\zeta_\rat(s) = L(s,\chi)$. 

Now further suppose that $q=p$ and that the Frobenius order is
maximal, i.e., that $\integ[T]/f_{X/\ff_q}(T) \iso \calo_K$.  Then
Gekeler shows with an explicit calculation that for \emph{each} prime $\ell$, $\nu_\ell^{\mathrm{
    Gek}}([X,\lambda]) = L_\ell(1,\chi)$, and thus $\prod_\ell
\nu_\ell^{\mathrm{Gek}}([X,\lambda]) = L(1,\chi)$.

\subsubsection{Abelian varieties with maximal Frobenius order}

Similarly, suppose $[X,\lambda]$ is an ordinary abelian surface with
$\End(X)\otimes\rat$ a cyclic quartic extension of $\rat$, and further suppose that the
Frobenius order is maximal.  In \cite{achterwilliams15},
the authors define a local term $\nu^{{\mathrm AW}}_\ell([X,\lambda])$,
and show that $\prod_\ell \nu^{{\mathrm AW}}_\ell([X,\lambda]) =
\zeta_K(1)/\zeta_{K^+}(1)$.  This observation has been extended to
certain abelian varieties of prime dimension
\cite[Prop.~8.1]{gerhardwilliams19}.

\subsection{$\nu_p$}

Since the multiplier $\eta(\gamma_0)$ of Frobenius is $q$,
$\gamma_0$, while an element of $M(\integ_p)\subset G(\rat_p)$, is
never an element of $G(\integ_p)$.  Nonetheless, if the isogeny class
is \emph{ordinary}, it is possible to transfer part of the work in
calculating $\nu_p([X,\lambda])$ from $M(\integ_p)$ to $G(\integ_p)$,
as follows.

Suppose $X$ is ordinary.  Then its $p$-divisible group splits
\emph{integrally} as $X[p^\infty] = X[p^\infty]^\tor\oplus
X[p^\infty]^\et$. (In general, the slope filtration only
exists up to isogeny, as in Lemma \ref{lem:pplussplits}.)
Therefore, there exists a decomposition $V_{\integ_p} =
V_{\integ_p}^\et \oplus V_{\integ_p}^\tor$ into maximal isotropic
summands stable under $\gamma_0$,
where $\alpha_0 :=
\gamma_0\rest{V_{\integ_p}^\et}\in \End(V_{\integ_p}^\et)$ is invertible;
and the polarization induces an isomorphism $V_{\integ_p}^\tor$ with
the dual of $V_{\integ_p}^\et$, such that $\gamma_0\rest{V_{\integ_p}^\tor} =
q(\alpha_0\transpose)\inv$.   This can also be proved directly
through linear algebra.  Indeed, let $\beta_0 = q\gamma_0\inv$.  Then
$V_{\integ_p}^\et = \cap_n \gamma_0^{\circ n}(V_{\integ_p})$, while
$V_{\integ_p}^\tor = \cap_n \beta_0^{\circ n}(V_{\integ_p})$.

\begin{lemma}
\label{lem:isodecomp}
For $n$ and $d$ sufficiently large and $\gamma \in M(\integ_p/p^n)$,
the following conditions are 
equivalent:
\begin{alphabetize}
\item $\gamma \twiddle_{M(\integ_p/p^n)_d} \gamma_0 \bmod p^n$;
\item there exists some $\til \gamma \in M(\integ_p)$ such that $\til
  \gamma \bmod p^n = \gamma$ and $\til \gamma \twiddle_{G(\rat_p)}
  \gamma_0$;
\item $\gamma$ stabilizes a decomposition $V_{\integ_p/p^n} \iso
  V_{\integ_p/p^n}^+ \oplus V_{\integ_p/p^n}^-$ into maximal
  isotropic subspaces, and there exists an isomorphism $\iota:
  V_{\integ_p/p^n}^+\to V_{\integ_p/p^n}^\et$ such that
  $\iota^*\alpha_0 = \gamma \rest{V_{\integ_p/p^n}^+}$.
\end{alphabetize}
\end{lemma}

\begin{proof}
The equivalence of (a) and (b) is Lemma \ref{lem:hensel}.  For the
equivalence of (a) and (c), use the argument above to show that any
such $\gamma$ induces an appropriate decomposition of
$V_{\integ_p/p^n}$.
\end{proof}

Therefore, if $\alpha_0\bmod p$ is regular, we obtain a version of
Lemma \ref{lem:zetafactor} at $p$.

\begin{corollary}
Suppose $[X,\lambda]$ is ordinary and $\ord_p \disc(f_{X/\ff_q}(T)) =
e \cdot g(g-1)$.  Then
\[
\nu_p([X,\lambda]) = \frac{\zeta_{K,p}(1)}{\zeta_{K^+,p}(1)}.
\]
\end{corollary}

Note that we always have $q^{g(g-1)} | \disc(f_{X/\ff_q}(T))$.  The
case $g=1$ also follows from the explicit calculation in \cite[Thm.~4.4]{gekeler03}.

\begin{proof}
Define $\epsilon_0 \in \SP(V_{\integ_p})$ by
$\epsilon_0\rest{V_{\integ_p}^\et} = \alpha_0$ and
$\epsilon_0\rest{V_{\integ_p}^\tor} = (\alpha_0\transpose)\inv$.  

The argument of Lemma \ref{lem:isodecomp} shows that, for sufficiently
large $d$ and $n$, both $\#C_{(d,n)}(\gamma_0)$ and
$\#C_{(d,n)}(\gamma_0)$ are given by the product of the number of
decompositions $V_{\integ_p/p^n} = V_{\integ_p/p^n}^+\oplus
V_{\integ_p/p^n}^-$ into maximal isotropic subspaces, and the number
of $\alpha \in \End(V_{\integ_p/p^n})$ with $\alpha
\sim_{\End(V_{\integ_p/p^n})_d} \alpha_0$.  In particular,
$\#C_{(d,n)}(\gamma_0) = \#C_{(d,n)}(\epsilon_0)$.  

The regularity hypothesis implies that $\epsilon_0\bmod p$ is regular,
and the result follows from Lemma \ref{lem:zetafactor}.
\end{proof}

\subsection{Explicit examples}

\subsubsection{$g=1$}

Consider the elliptic curve $E/\ff_7$ with affine equation $y^2 =
x^3+x+1$.  Its Frobenius polynomial is $f_E(T) = T^2-3T+7$, which has
discriminant $-19$, a fundamental discriminant.  So the order
generated by the Frobenius endomorphism is the ring of integers in
$K:= \rat(\sqrt{-19})$; using Magma, we numerically estimate $\prod_\ell
\nu_\ell(E/\ff_7) = L(1,\left( \frac{-19}{\bullet} \right))$ as $\approx
0.72073$. Continuing to work numerically, we have $\nu_\infty(E/\ff_7)
= \oneover{2\pi}\sqrt{4 - \frac{3^2}{7}} \approx
0.2622$.  Since $T = \res_{K/\rat}\gp_m$ we have $\tau_T = 1$, and thus $\tau_T
 \sqrt 7 \nu_\infty(E/\ff_7)\prod_\ell \nu_\ell(E/\ff_7)\approx 0.5000$.

This reflects the easily verified arithmetic statement that the only
elliptic curve over $\ff_7$ with trace of Frobenius $3$ is $E$ itself;
and $\aut(E) \iso \calo_K\units = \st{\pm 1}$, so that the weighted size of this
isogeny class is $\wnum I([E],\ff_7) = \oneover 2$.  (In modest
contrast, \cite{gekeler03} assigns weight $2/\#\aut(F)$ to an elliptic
curve $F$; this is reflected in the fact that
$\nu_\infty^{\mathrm{Gek}}([E],/\ff_7) = 2 \nu_\infty([E],\ff_7)$.)

\subsubsection{$g=4$}
\label{ex:g4}

Consider the $3$-Weil polynomial 
\[
f(T) = T^8 - 6T^7 + 13T^6 - 10T^5
+ T^4 - 30T^3 + 117T^2 - 162T + 81.
\]
It turns out that there is a unique principally polarized abelian
fourfold $(X,\lambda)$ over $\ff_3$ with characteristic polynomial equal to $f(T)$.
(This is a single datapoint in a census of isogeny classes which will
soon be integrated into the LMFDB.)

Let $K = \rat[T]/f(T)$. One readily checks that $\disc(f(T))/\disc(K)
= 3^{4(4-1)}$, and so $\nu_\ell([X,\lambda]) =
\zeta_{K,\ell}(1)/\zeta_{K^+,\ell}(1)$ for all finite $\ell$,
including $\ell = p$. Again, we numerically estimate $\prod_\ell
\nu_\ell([X,\lambda]) = \lim_{s \to 1^+}
\frac{\zeta_K(s)}{\zeta_{K^+}(s)} \approx 0.871253$ and
$\nu_\infty([X,\lambda])  \approx 0.000111808$.  The field $K$ is
Galois over $\rat$, with group $\Gal(K/\rat) \iso \integ/4\oplus
\integ/2$, and the Tamagawa number of the torus $T$ is $2$ (see \ref{seq:sage}).  Our formula numerically
yields 
\[
\wnum I([X,\lambda],\ff_3)  = 3^{3}\tau_T \nu_\infty([X,\lambda],\ff_3)
\lim_{s\to 1} \frac{\zeta_K(s)}{\zeta_{K^+}(s)}\approx 0.050000.\]
This reflects the fact that the torsion group of  $\calo_K\units$, and
thus $\aut([X,\lambda])$, has order $20$.

The referee points out that it is possible to take advantage of other
recent work on isogeny classes, which appeared essentially
simultaneously with
or since the first preprint of this work was made available, to
independently verify 
these assertions without recourse to the LMFDB.  For example, our
characterization of $K$ shows that in fact $K = \rat(\zeta_{20})$; and
the discriminant calculation shows that $\End(X) =
\integ[\zeta_{20}]$.  Since $h(K) = h(K)/h(K^+)=1$, one can use
\cite[Cor.~1.2]{howevariation} or  \cite[Thm.~1.1]{gsy20} to conclude
that $\# I([X,\lambda],\ff_3) = 1$.

The state of the art on calculation of Tamagawa numbers has also
improved; see \cite{lyy21,rud20}.

\subsection{Level structure}
The Langlands-Kottwitz formula \eqref{eqlk} is actually written for abelian varieties with
arbitrary level structure, and thus a version of our main formula is available in the context of abelian varieties with level structure, too.

\subsubsection{Product formula}
Let $\level \subset
G(\hat\integ_f^p)$ be an open compact subgroup.  There is a notion of principally polarized abelian variety with level $\level$ structure; let $\stack A_{g,\Gamma}$ be the corresponding Shimura variety.  If
$(X,\lambda,\alpha)\in \stack A_{g,\Gamma}(\ff_q)$ is a principally polarized abelian variety
with level $\level$-structure, then the size of its isogeny class in
this category is given by the Kottwitz formula, except that the
integrand in the adelic orbital integral is replaced with
$\one_{\level}$.

We make the definition
\[
\nu_\ell([X,\lambda,\alpha]) = \lim_{d\to\infty} \lim_{n\to\infty} \frac{\#(C_{(d,n)}(\gamma_0)\cap \pi_n(\level_\ell))}{\#G(\integ_\ell/\ell^n)/\#\A_G(\integ_\ell/\ell^n)}.
\]

The analogue of Corollary \ref{lem:stein_oireprise} holds, and states that
there exists $d(\gamma_0)$ such that
\[
O^{\geom}_{\gamma_0}(\one_{\level_\ell})=\lim_{n\to
  \infty}\frac{\vol_{{\abs{d\omega_G}}}(\til
  C_{(d(\gamma_0),n)}(\gamma_0) \cap \Gamma_\ell)}{\vol_{\abs{d\omega_{A}}}(\til
  U_n(\gamma_0))}.
\]

The calculations at $p$ and $\infty$, as well as the global volume term, are unchanged, and we find that 
\begin{equation}
\tilde \# I([X,\lambda,\alpha], \F_q)= q^{\frac{g(g+1)}4}\tau_T \nu_{\infty}([X,\lambda])\prod_\ell \nu_\ell([X,\lambda,\alpha]).
\end{equation}

\subsubsection{Principal level structure}
Fix a prime $\ell_0$, and define $\level(\ell_0) = \prod_\ell
\level(\ell_0)_\ell$ by
\[
\level(\ell_0)_\ell = \begin{cases}
G(\integ_\ell) & \ell\not = \ell_0 \\
\ker(G(\integ_{\ell_0}) \to G(\integ_{\ell_0}/\ell_0)) & \ell = \ell_0.
\end{cases}
\]

Then $\stack A_{g,\level(\ell_0)}$ is the moduli space of abelian varieties equipped with a full principal level $\ell_0$-structure.

For example, to fix ideas, suppose that $g=1$ and that $\ell_0\not = p$, and let $a$ satisfy $\abs a \le 2 \sqrt q$, $p\nmid a$ and $\ell_0 |\!| (a^2-4q)$; we consider the set of elliptic curves with characteristic polynomial of Frobenius $f(T) = T^2-aT+q$.  Then some, but not all, elements of the corresponding isogeny class admit a principal level $\ell_0$-structure (see, e.g., \cite{achterwong}).

Let $(X,\lambda,\alpha)$ be an elliptic curve over $\ff_q$ with trace of Frobenius $a$ and full level $\ell_0$-structure $\alpha$.  We may explicitly compute $\nu_{\ell_0}([X_0,\lambda,\alpha])$ as follows.  Let $\chi_{\ell_0} = \left(\frac{\cdot}{\ell_0}\right)$ be the quadratic character modulo $\ell_0$.

\begin{lemma}
We have
\[
\nu_{\ell_0}([X,\lambda,\alpha]) = \oneover{\ell_0^2} \oneover{1-\chi_{\ell_0}(\disc(f)/\ell_0^2)/\ell_0}.
\]
\end{lemma}

\begin{proof}
Let $\gamma_0 = \gamma_{X,\ff_q,\ell_0}$ be a Frobenius element for $X$ at $\ell_0$.  By hypothesis, $\gamma_0 = 1 + \ell_0 \beta_0$ for some $\beta_0 \in \operatorname{Mat}_2(\integ_{\ell_0})$.  Since $\ell_0^2$ is the highest power of $\ell_0$ dividing $\disc(f)$, we in fact have $\beta_0 \in \gl_2(\integ_{\ell_0})$, and $\beta_0$ is regular $\bmod \ell_0$, i.e., $\pi_1(\beta_0)$ is regular.

Suppose that $\gamma \in \Gamma_{\ell_0}$ satisfies $\gamma \twiddle_{G(\rat_{\ell_0})}\gamma_0$.  Then $\gamma = 1 + \ell_0\beta$ for some $\beta \in \gl_2(\integ_{\ell_0})$ which is regular $\bmod \ell_0$, and direct calculation shows $\beta\twiddle_{G(\rat_{\ell_0})} \beta_0$.  Lemma \ref{lem:gqvsgz} then shows that $\beta\twiddle_{G(\integ_\ell)} \beta_0$.  

Consequently, for any $d \ge 0$ and any $n\ge 2$, we have bijections between the following sets:
\begin{align*}
    \st{\gamma \in G(\integ_{\ell_0}/\ell_0^n) : \gamma \twiddle_{M(\integ_{\ell_0}/\ell_0^n)_d} \pi_n(\gamma_0)}; \\
   \st{\beta \in G(\integ_{\ell_0}/\ell_0^{n-1}) : \beta \twiddle_{M(\integ_{\ell_0}/\ell_0^{n-1})_d} \pi_{n-1}(\beta_0)}; \\
     \text{ and }\st{\beta \in G(\integ_{\ell_0}/\ell_0^{n-1}) : \beta \twiddle_{G(\integ_{\ell_0}/\ell_0^{n-1})} \pi_n(\beta_0)}. \\
\end{align*}

Therefore,
\begin{align*}
    \nu_{\ell_0}([X,\lambda,\alpha]) &= 
    \frac{\# G(\integ_{\ell_0}/\ell_0^{n-1})/\#G_{\beta_0}(\integ_{\ell_0}/\ell_0^{n-1})}
    {\#G(\integ_{\ell_0}/\ell_0^n)/\#\A_G(\integ_{\ell_0}/\ell_0^n)}\\
    &=
     \frac{\# G(\integ_{\ell_0}/\ell_0^{n-1})}{\#G(\integ_{\ell_0}/\ell_0^n)}
     \frac{\#\A_G(\integ_{\ell_0}/\ell_0^n)}{\#G_{\beta_0}(\integ_{\ell_0}/\ell_0^{n-1})} \\
     &=
     \oneover{\ell_0^4}\frac{\ell_0^2 \#\A_G(\integ_{\ell_0}/\ell_0^{n-1})}{\#G_{\beta_0}(\integ_{\ell_0}/\ell_0^{n-1})}\\
     &= \oneover{\ell_0^2} \frac{\#\A_G(\integ_{\ell_0}/\ell_0)}{\#G_{\beta_0}(\integ_{\ell_0}/\ell_0)} \\
     &= \oneover{\ell_0^2} \oneover{1 - \chi_{\ell_0}(\disc(f)/\ell_0^2)/\ell_0}.
\end{align*}

\end{proof}
\section{$\gl_2$ reconsidered}

In \cite{achtergordon17}, we essentially treated the $g=1$ case of the
present paper.  Unfortunately, a simple algebra error --
$\nu_\infty^{\mathrm{AG}}([X,\lambda]) = \frac{2}{\pi}
\sqrt{\abs{D(\gamma_0)}}$ (\S \ref{subsub:nuinfg1}), in spite of the claims of the penultimate
displayed equation \cite[p.20]{achtergordon17} -- masked certain
mistakes involving the calculations at $p$.  We take the opportunity
to correct these mistakes.  The reader pleasantly unaware of these
issues with \cite{achtergordon17} may simply view the present section as an
explication of our technique in the special case where $g=1$, and thus
$G = \gl_2$.

Note that the definition \cite[(2-6)]{achtergordon17} could have been replaced with
a criterion involving characteristic polynomials, e.g.,
\begin{equation*}
\nu_p(a,q) = \lim_{n\ra\infty} \frac{\#\st{ \gamma \in
    \mat_{2}(\integ_p/p^n): f_\gamma \equiv f_{\gamma_0} \bmod
    p^n}}{\#G(\integ_p/p^n)/\#A(\integ_p/p^n)}.
\end{equation*}

\subsection{Assertions at $p$}

There are two problematic claims in \cite{achtergordon17}:
\begin{enumerate}

\item
\label{it:blunder1}
 For the test function $\one_{G(\integ_\ell)}$ at $\ell\not = p$, we have
\[
\nu_\ell(a,q) = \frac{\ell^3}{\#\SL_2(\ff_\ell)}
\calo_{\gamma_0}^{\geom}(\one_{G(\integ_\ell)}).
\]
It is claimed in \cite[Lemma 3.7]{achtergordon17} that the same is
true for $\ell =p$, where the test function $\phi_q$ is the
characteristic function of $G(\integ_p)\begin{pmatrix} q&0\\0&1\end{pmatrix}G(\integ_p)$.

\item\label{it:blunder2} In \cite[Appendix]{achtergordon17}, revisiting the calculation
  \cite[(3.30)]{langlands-frenkel-ngo}, we assert that.
\begin{equation}\label{eq:blunder2}
\mu_{\gamma,\ell}^\geom=\sqrt{\abs{D(\gamma)}_\ell}\frac{\vol_{\abs{\omega_G}_\ell}(G(\integ_\ell))}
{\vol_{\abs{\omega_T}_\ell}(T^\circ(\integ_\ell))}\bar{\mu}_{T\backslash G,\ell}.
\end{equation}
This is valid for $\ell\not = p$, but requires correction at $p$, because our Chevalley-Steinberg map is not exactly the same as the map in \cite{langlands-frenkel-ngo}.
\end{enumerate}

\subsection{From point-counts to measure}

In  \eqref{it:blunder1}, one exploits the fundamental fact that point
counts $\bmod \ell^n$ converge to volume with respect to the
Serre-Oesterl\'e measure $\mu^{\serre}$.  In the case where $\ell = p$,
however, the ambient space is $\mat_2(\integ_p)$, rather
than (its open subset) $G(\integ_p)$.  Thus, the volume of the set
$V_n$ of \cite{achtergordon17} should be computed with respect to 
$\mu^\serre_{\mat_2}$, which is the usual measure on the $4$-dimensional affine space.
We recall that for $\gamma\in \GL_2(\Q_p)$, 
$d\mu^{\serre}_{\GL_2}(\gamma)=|\det(\gamma)|^{-2}d\mu^{\serre}_{\mat_2}(\gamma)$.   

Moreover, one should be using the invariant measure $dx \wedge
\frac{dy}{\abs y}$ on the Steinberg base $\A_{\GL_2} \iso \aff^1 \times \gp_m$,
rather than the measure pulled back from $\aff^1 \times \aff^1$. Then
the measure of a radius $p^{-n}$-neighborhood $U_n$ of $(a,q)$ in $A$
is $p^{-2n}/\abs{\det(\gamma_0)} = p^{-2n}/q\inv$, and we find
\begin{multline*}
\nu_{p, n}(a,q)=\frac{p^{3}}{\#\SL_2(\F_p)} \frac{\abs{\det(\gamma_0)}^2 \vol_{\mu_{GL_2}}(V_n(\gamma_0))}{\abs{\det(\gamma_0)}\vol_{\A^1\times \G_m}(U_n)}
\\ = \abs{\det(\gamma_0)}\frac{p^{3}}{\#\SL_2(\F_p)}O^{\geom}_{\gamma_0}(\phi_q)
= q^{-1} \frac{p^{3}}{\#\SL_2(\F_p)}O^{\geom}_{\gamma_0}(\phi_q). 
\end{multline*}
(This differs from the assertion of
\cite[Lemma 3.7]{achtergordon17} by a factor of $q$.)

\subsection{From geometric measure to canonical measure}

Since orbital integrals of rational-valued functions with respect to
the canonical measure are rational, while $\sqrt{\abs{D(\gamma_0)}_p}
= \sqrt q$, the assertion of \eqref{it:blunder2} cannot hold at $\ell
= p$.  

While the relation between the geometric measure and the canonical
measure that we rely on is correct for a semisimple group, it needs a
correction factor for a reductive group. This part is completely
general for all reductive $G$, and is discussed in detail in
\cite{jg:singapore}; the correct formula is stated in Proposition
\ref{prop:ratio}.  In particular, the correct calculation at $p$ is
\[
\mu_{\gamma,p}^\geom=\abs{\eta(\gamma)}_p^{-\frac{g(g+1)}4}\sqrt{\abs{D(\gamma)}_p}\frac{\vol_{\abs{\omega_G}_p}(G(\integ_p))}
{\vol_{\abs{\omega_T}_p}(T^\circ(\integ_p))}\bar{\mu}_{T\backslash G,p},
\]
where $\eta(\gamma)$ is the multiplier of $\gamma$.

\appendix
\section{By  Wen-Wei Li and Thomas R\"ud}\label{sec:tamagawa}

We compute the Tamagawa numbers of some anisotropic tori in $\mathrm{GSp}_{2g}$ and $\mathrm{Sp}_{2g}$ associated with a single Galois field extension (see section \ref{subsec:torus}), and present a partial result towards the general case that illustrates the difficulties.

Recall the setup of $\ref{subsec:torus}$ in the case of a single
Galois extension. Let $K\supset K^+\supset \Q$ be a 
tower of field extensions with $K$ Galois, such that $[K:K^+]=2$ and $[K^+:\mathbb{Q}]=g$. We define 
\[T^\der = \mathrm{Ker}\left( \res_{K/\Q}(\gp_m) \overset{N_{K/K^+}}{\longrightarrow} \res_{K^+/\Q}(\gp_m)\right)= \res_{K^+/\Q}\res_{K/K^+}^{(1)}(\gp_m)\subset \mathrm{Sp}_{2g}\ ,\] and 
\[T = \mathrm{Ker}\left( \gp_m\times_{\mathrm{Spec}(\Q)}  \res_{K/\Q}(\gp_m)\overset{(x,y)\mapsto x^{-1}N_{K/K^+}(y)}{\longrightarrow} \res_{K^+/\Q}(\gp_m)\right)\subset \mathrm{GSp}_{2g}.\]
They fit in the short exact sequence 
\[
\xymatrix{
1\ar[r] &  T^\der \ar[r]& T\ar[r] &  \mathbb{G}_m\ar[r] &  1}
.\]

Also, recall that we have been using $\tau_T$ (resp $\tau_{T^\der}$) to denote the Tamagawa numbers $\tau_\Q(T)$ and $\tau_\Q(T^\der)$. We will show the following. 

\begin{proposition}\label{prop:appendix_tamagawaTder} One has $\tau_{T^\der} = \tau_{K^+}(\resone_{K/K^+}\gp_m) = 2$. 
\end{proposition}

For the case of $T$, the result varies with the extension. 

\begin{proposition}\label{prop:appendix_tamagawaT} Assume that $K$ is a Galois CM field and $K^+$ is its maximal totally real subfield. Then we have $\tau_T\leq 2$, and moreover : 
\begin{itemize}
  \item If $g$ is odd then $\tau_T=1$.
  \item If $K/\Q$ is cyclic, then $\tau_T=1$.
  \item If $g=2$ then $\tau_T =1$ when $\mathrm{Gal}(K/\Q)\cong \Z/4\Z$ and $\tau_T=2$ when $\mathrm{Gal}(K/\Q)\cong (\Z/2\Z)^2$. 
\end{itemize}
\end{proposition}

More results and details will appear in the second author's forthcoming thesis.

We base our approach the following formula  of Ono. 
\begin{theorem}[\cite{ono_tamagawa_tori}]\label{thm:ono_tamagawa_tori}
  Let $\mathbf{T}$ be an algebraic torus defined over a number field
  $F$ and split over some Galois extension $L$. Then its Tamagawa
  number can be computed as  \[\tau_{F}(\mathbf{T}) =
  \frac{|H^1(L/F,\mathbf{X}^\star(\mathbf{T}))|}{|\Sh^1(\mathbf{T})|}.\]
  Here $\mathbf{X}^\star(\mathbf{T})$  denotes the character lattice
  of $\mathbf{T}$. The symbol $\displaystyle\Sh^1(\mathbf{T})$ denotes
  the corresponding Tate-Shafarevich group defined
  by \begin{equation}\label{def:sha}\Sh^i(\mathbf{T}) =
    \mathrm{Ker}(H^i(L/F,\mathbf{T})\rightarrow \prod_{v}H^i(L_w/F_v,
    \mathbf{T})),\end{equation} where $v$ runs over the primes of $F$
  and $w$ is a prime of $L$ with $w|v$.  
\end{theorem}

Our approach is to do the computation on the level of character lattices. 
A very important  consequence of Tate-Nakayama duality theorem (see \cite[Theorem 6.10]{platonov-rapinchuk:AlgGroupsAndNT}) is that for a torus $\mathbf{T}$ as in the previous theorem, the Pontryagin dual of $\Sh^1(\mathbf{T})$ is isomorphic to $\Sh^2(\mathbf{X}^\star(\mathbf{T}))$, so it suffices to compute $|\Sh^2(\mathbf{X}^\star(\mathbf{T}))|$.

The proof of proposition \ref{prop:appendix_tamagawaTder} is given in the next section. The proof of proposition \ref{prop:appendix_tamagawaT} occupies sections \ref{seq:lattice} to \ref{seq:g2}. In section \ref{seq:sage} we present an example not covered by proposition \ref{prop:appendix_tamagawaT}. In section \ref{seq:general} we present a computation for the numerator that illustrates the difficulties that arise for a general torus (not assuming that $T$ is contructed from a single field).

\subsection{Computation of $\tau_{T^\der}$}
\label{seq:computation_tder}
We write the proof of proposition \ref{prop:appendix_tamagawaTder} using Theorem \ref{thm:ono_tamagawa_tori}. Since the Tamagawa number is preserved by restriction of scalars we have $\tau_{T^\der}  = \tau_\Q(\res_{K^+/\Q}\resone_{K/K^+}\gp_m) =\tau_{K^+}(\resone_{K/K^+}\gp_m) $. The cohomology of  the characters of a norm $1$ torus is obtained by a classic computation that one can see for instance in the proof of the Hasse norm theorem in \cite[Theorem 6.11]{platonov-rapinchuk:AlgGroupsAndNT}. We have 
\[\hat{H}^{i}(K/\Q, \mathbf{X}^\star(T^\der)) =\hat{H}^{i}(K/K^+, \resone_{K/K^+}\gp_m) = \hat{H}^{i+1}(\Z/2\Z,\Z) =\left\lbrace \begin{array}{l} \Z/2\Z\text{ if }i\text{ is odd}\\ \{0\}\text{ if }i \text{ is even} \end{array} \right.  .\]
In particular, $|\Sh^1(T^\der)|=|\Sh^2(\mathbf{X}^\star(T^\der))| \leq |H^2(\mathbf{X}^\star(T^\der))| =1 $. We conclude $\tau_{T^\der}=\frac{2}{1}=2.$

This proves proposition \ref{prop:appendix_tamagawaTder}.
\subsection{Computation of the first cohomology group of the character lattice}
\label{seq:lattice}
From now on, we focus on the proof of proposition \ref{prop:appendix_tamagawaT}, and therefore we will assume that $K$ is a CM-field with $K^+$ its maximal totally real subfield.
Let $\iota$ be the nontrivial element of $\mathrm{Gal}(K/K^+)$, and let $\Gamma,\Gamma^+$ denote respectively the Galois groups of $K/\mathbb{Q}$, and $K^+/\mathbb{Q}$. Note that $K^+$ is indeed Galois over $\mathbb{Q}$ by virtue of $K$ being a CM-field. The torus $T$ arises as the subtorus of $\res_{K/\mathbb{Q}}(\mathbb{G}_m)$ with the set of $\Q$-points consisting of elements $x\in K^\times$ such that $x\iota(x) \in \mathbb{Q}$ and $T^\der(\Q)$ is the set of elements $x\in K^\times$ such that $x\iota(x)=1$. 
We have the following exact sequence of finite groups: 
\begin{equation}\label{eq:seq_group}
\xymatrix{
1 \ar[r]& \left\langle\iota \right\rangle\cong \Z/2\Z \ar[r]& \Gamma
\ar[r] &  \Gamma^+\ar[r]&  1}.
\end{equation}
For each $\sigma\in \Gamma^+$ fix a preimage $\hat{\sigma}\in \Gamma$. We get the description of $\mathbf{X}^\star(\res_{K/\Q}(\gp_m)) = \Z[\Gamma]$ as the set of $\Z$-linear combinations of $\hat{\sigma}$ and $\hat{\sigma}\iota$ with $\sigma\in \Gamma^+$.

The embedding of $T$ in $\res_{K/\Q}(\gp_m)$ gives us a surjective map $\mathbf{X}^\star(\res_{K/\Q}(\gp_m))\rightarrow \mathbf{X}^\star(T)$.  For $\chi = \sum_{\sigma\in \Gamma^+}a_\sigma\hat{\sigma} +\sum_{\sigma\in \Gamma^+}b_\sigma\hat{\sigma}\iota  \in\mathbf{X}^\star(\res_{K/\Q}(\gp_m))$ and $t\in T(\Q)$, we have 

\begin{align*}
\chi(t) &= \prod_{\sigma\in\Gamma^+}\hat{\sigma}(t)^{a_\sigma}\prod_{\sigma\in \Gamma^+}\hat{\sigma}(\iota(t))^{b_\sigma}\\
    &= \prod_{\sigma\in\Gamma^+}\hat{\sigma}(t)^{a_\sigma}\prod_{\sigma\in \Gamma^+}\hat{\sigma}\left(\lambda t^{-1}\right)^{b_\sigma}\qquad \text{ where }t\iota(t)=\lambda\in \Q^\times\\
   &=\lambda^{\sum_{\sigma\in \Gamma^+}b_\sigma} \prod_{\sigma\in\Gamma^+}\hat{\sigma}(t)^{a_\sigma-b_\sigma}.\\
\end{align*}
We get the descriptions : 

\begin{align}\label{eq:char_description}
\mathbf{X}^\star(T) &= \Z[\Gamma] / \left\{\sum_{\sigma\in \Gamma^+}a_\sigma\hat{\sigma}+ \sum_{\sigma\in \Gamma^+}b_\sigma\hat{\sigma} \iota  : a_\sigma=b_\sigma\text{ for }\sigma\in \Gamma^+\text{ and} \sum_{\sigma\in {\Gamma^+}}b_{{\sigma}}=0\right\}\\ 
&= \Z[\Gamma] /L, 
\end{align}
where $L=\left\{\sum_{\sigma\in \Gamma^+}a_{{\sigma}}\hat{\sigma} (1+\iota)  : \sum_{\sigma\in \Gamma^+}a_{{\sigma}}=0\right\}$. For $T^\der$, we have $\lambda=1$ so we recover 
\begin{equation}\mathbf{X}^\star(T^\der) = \Z[\Gamma]/\{x=\iota(x)\} = \Z[\widehat{\Gamma^+}]\otimes \Z[\iota]/\left\langle 1+\iota\right\rangle  = \mathbf{X}^\star(\mathbf{R}_{K^+/\Q}\mathbf{R}^{(1)}_{K/K^+}(\mathbb{G}_m)).\end{equation}

In order to compute $H^1(K/\Q,\mathbf{X}^\star(T))$ we use the inflation-restriction exact sequence, which one can find in \cite[Proposition 3.3.14 p.65]{gille_csa}. To simplify notations, let $\Lambda = \mathbf{X}^\star(T) = \Z[\Gamma]/L$ as in (\ref{eq:char_description}).

The inflation-restriction exact sequence associated with the short exact sequence (\ref{eq:seq_group}) takes the form 
\begin{equation}\label{eq:infl_res}
\xymatrix@C-0.8em{
0\ar[r] &  H^1(\Gamma^+,\Lambda^{\Z/2\Z}) \ar[r]&  H^1(\Gamma,\Lambda)
\ar[r] &  H^1(\Z/2\Z,\Lambda)^{\Gamma^+}\ar[r] &
H^2(\Gamma^+,\Lambda^{\Z/2\Z})\ar[r] &  H^2(\Gamma,\Lambda)}.\end{equation}

\begin{lemma}\label{lem:spec_h1}
The sequence (\ref{eq:infl_res}) can be rewritten as 
\begin{equation}\label{eq:spec_seq}
\xymatrix{
0\ar[r]& 0\ar[r]& H^1(\Gamma,\Lambda)\ar[r]&
(\Z/2\Z)^{\frac{1+(-1)^g}{2}}
\ar[r]& {\Gamma^+}^{\mathrm{ab}}\ar[r]& H^2(\Gamma,\Lambda)}.\end{equation}
In particular, $\tau_T\leq |H^1(\Gamma,\mathbf{X}^\star(T))|\leq 2$.
\end{lemma}
\begin{proof}
Let $x\in \Z[\Gamma]$, and let $[x]$ denote its class in $\Lambda$. Clearly $[x]$ is fixed by $\iota$ if and only if $x-\iota x\in L$, and since every element of $L$ is fixed by $\iota$, then so must $x-\iota x$ which forces $x=\iota x$. Therefore, 
\[\Lambda^{\Z/2\Z} = \{\sum_{\sigma\in \Gamma^+}a_\sigma\hat{\sigma} (1+\iota)\} / \left\langle \sum_{\sigma\in \Gamma^+}a_\sigma\hat{\sigma} (1+\iota) : \sum_{\sigma\in \Gamma^+}a_\sigma = 0\right\rangle\cong \Z[\Gamma^+]/I,\]
where $I$ is the augmentation ideal of $\Z[\Gamma^+]$, i.e. the subspace of sum-zero vectors. Further, observe that $\Z[\Gamma^+]/I\cong \Z$ as $\Gamma^+$-modules where $\Z$ has trivial $\Gamma^+$-action (by definition of $I$).  We get that 
\[H^1(\Gamma^+,\Lambda^{\Z/2\Z}) \cong H^1(\Gamma^+,\Z) = \mathrm{Hom}(\Gamma^+,\Z) = \{0\}. \]
Also, using the sequence 
\begin{equation}\label{seq:z_in_q}
\xymatrix{
 0\ar[r]& \Z\ar[r]& \Q\ar[r]& \Q/\Z\ar[r]& 0},\end{equation}
since the middle term is uniquely divisible hence cohomologically trivial, one has 
\[H^2(\Gamma^+,\Z)\cong H^1(\Gamma^+,\Q/\Z) =
  \mathrm{Hom}(\Gamma^+,\Q/\Z);\]
the last term is noncanonically isomorphic to${\Gamma^+}^{ab}$.

The only term left to compute is $ H^1(\Z/2\Z,\Lambda)^{\Gamma^+}$.  As a $\Z/2\Z$-module,  we can write $\Lambda = \Z^g\oplus \Z^g/L$ where $L= \{(a,a) : a=(a_1,\dots,a_g)\ \sum a_i=0\}$, and $\Z/2\Z$ acts as $(a,b)\mapsto (b,a)$. Therefore, we have 
\[
\xymatrix{
0\ar[r]& L \ar[r]& \Z^g\oplus \Z^g = \Z[\Z/2\Z]^g\ar[r]& \Lambda\ar[r]& 0}.\]
Since the middle term is cohomologically trivial as a $\Z/2\Z$-module,
and $\Z/2/Z$ acts trivially on $L$, we have
\[H^1(\Z/2\Z,\Lambda) \cong H^2(\Z/2\Z,L)\cong
\hat{H}^0(\Z/2\Z,L)\cong L/2L.\] To compute $L/2L$, we view $L$ as a submodule of $\Z^g$ of zero-sum
elements. Recall that by construction of $L$ as a group algebra,
$\Gamma^+$ acts transitively on $L$.

Let $\mathbf{a}= (a_1,\dots,a_g)\in L$. We want to compute $(L/2L)^{\Gamma^+}$, and for that we reason on the parity of $a_i$'s. 
\begin{itemize}
    \item If all $a_i$ are even, then $\mathbf{a}=2{\mathbf a'}$ and $\mathbf{a'}\in L$, hence $\mathbf{a}\in 2L$. 
    \item If $\mathbf{a}$ has $a_i$ even and $a_j$ odd, considering a permutation $\sigma\in \Gamma^+$ sending the $i$th coordinate to the $j$th, we have that the $j$th coordinate of $\mathbf{a}-\sigma \mathbf{a}$ is $a_j-a_i$ which is odd, hence $\mathbf{a}-\sigma \mathbf{a}\notin 2L$. 
    \item The last case to consider is when all $a_i$ are odd. In that
      case, $\sum_i a_i$ has the same  parity as $g$, so $\mathbf{a}\in L$ can
      only happen if $g$ is even. One can prove that every element of
      $L/2L$ has a representative of the form
      $\mathbf{a}=(a_1,\dots,a_g)\in \Z^g$ with $\sum_{i}a_i=0$ and
      $|a_i|\leq 1$. When $g$ is even and all $a_i$ are odd, the only
      possible such elements are vectors with half the coordinates
      being $-1$ and the other half $1$. Moreover all such vectors are
      in the same coset of $2L$ (one can permute the $\pm 1$
      coordinates by adding $\pm2$).   
\end{itemize}
This shows that $(L/2L)^{\Gamma^+}$ contains no nontrivial element when $g$ is odd, and only one when $g$ is even and conclude the proof.
\end{proof}

\begin{corollary}\label{cor:g_odd}
When $g$ is odd, one has $H^1(K/\Q,\mathbf{X}^\star(T)) = \{0\}$ and \[H^2(K/\Q,\mathbf{X}^\star(T))\cong {\Gamma^+}^\mathrm{ab}.\]
In particular, $H^2(K/\Q,\mathbf{X}^\star(T))$ has odd order, and so does  $\Sh^1(T)$, since it is dual to  $\Sh^2(\mathbf{X}^\star(T))\subset H^2(K/\Q,\mathbf{X}^\star(T))$.
\end{corollary}
\begin{proof}
The first equality comes directly from the sequence (\ref{eq:spec_seq}).

For the second equality, taking duals of the exact sequence (\ref{eq:diagT}), one gets 
\[
\xymatrix{
0\ar[r]&  \Z\ar[r]& \mathbf{X}^\star(T)\ar[r]&
\mathbf{X}^\star(T^\der)
\ar[r]& 0}. \]
The cohomology of this sequence gives us 
\[ 
\xymatrix{H^1(\Gamma,\mathbf{X}^\star(T))\ar[r]&
H^1(\Gamma,\mathbf{X}^\star(T^\der))
\ar[r]& H^2(\Gamma,\Z)\ar[r]& H^2(\Gamma,\mathbf{X}^\star(T))
\ar[r]& H^2(\Gamma,\mathbf{X}^\star(T^\der))}.\]

We computed the cohomology $H^i(\Gamma,\mathbf{X}^\star(T^\der)) = H^i(\Z/2\Z,\mathbf{X}^\star(\resone_{K/K^+}\gp_m))$ in section \ref{seq:computation_tder}.

We can plug in in $H^1(\Gamma,\mathbf{X}^\star(T)) = \{0\} = H^2(\Gamma,\mathbf{X}^\star(T^\der)) $,  $H^2(\Gamma,\Z)\cong  H^1(\Gamma,\Q/\Z) \cong  \Gamma^{ab}$, and $H^2(\Gamma,\mathbf{X}^\star(T^\der))\cong \Z/2\Z$, which gives us

\[\xymatrix{ 0\ar[r]& \Z/2\Z\ar[r]& \Gamma^{ab}\ar[r]& H^2(\Gamma,\mathbf{X}^\star(T))\ar[r]& 0},\]
as desired. 
\end{proof}

\begin{proposition}\label{prop:h1_split}
When the sequence (\ref{eq:seq_group}) splits, one has $H^1(K/\Q,\mathbf{X}^\star(T)) = \left\lbrace  \begin{array}{l} \{0\}\text{ if }g\text{ is odd}\\ \Z/2\Z\text{ if }g\text{ is even}   \end{array}\right.$. In particular, this gives an alternate proof of the triviality of $H^1(K/\Q,\mathbf{X}^\star(T))$ whenever $g$ is odd.
\end{proposition}
\begin{proof}
Since (\ref{eq:seq_group}) splits, one can write the inflation-restriction exact sequence associated with the short exact sequence 
\[\xymatrix{1\ar[r]& \Gamma^+\ar[r]& \Gamma\ar[r]& \Z/2\Z\ar[r]& 1}.\]
This gives us 
\begin{equation}\label{eq:infl_res_split}\xymatrix@C-0.9em{0\ar[r]& H^1(\Z/2\Z,\Lambda^{\Gamma^+})\ar[r]& H^1(\Gamma,\Lambda)\ar[r]& H^1(\Gamma^+,\Lambda)^{\Z/2\Z}\ar[r]& H^2(\Z/2\Z,\Lambda^{\Gamma^+})\ar[r]& H^2(\Gamma,\Lambda)}.\end{equation}

Since the sequence (\ref{eq:seq_group}) splits, we have $\Z[\Gamma]=\Z[\Gamma^+]\otimes\Z[\iota]$ as a $\Gamma$-module. We have $\Lambda = \Z[\Gamma^+]\otimes \Z[\iota]/I\otimes (1+\iota)$ where $I$ is the augmentation ideal of $\Z[\Gamma^+]$. Since $\Z[\Gamma^+]$ is an induced module, and $I\otimes (1+\iota)\cong I$ as a $\Gamma^+$-module, we get $\hat{H}^i(\Gamma^+,\Lambda) = \hat{H}^{i+1}(\Gamma^+,I)$. Now since $\Z= \Z[\Gamma^+]/I$ with $\Z$ seen as a trivial module, the same argument yields $\hat{H}^{i}(\Gamma^+,\Lambda)\cong \hat{H}^i(\Gamma^+,\Z)$. In particular, $H^1(\Gamma^+,\Lambda) = \{0\}$ so the sequence (\ref{eq:infl_res_split}) gives an isomorphism $H^1(\Gamma,\Lambda) \cong H^1(\Z/2\Z,\Lambda^{\Gamma^+})$.

Direct computations using that $\overset{+}{\sigma} := \sum_{\sigma\in
  \Gamma^+}\sigma$ spans the set of $\Gamma^+$-fixed elements of
$\Z[\Gamma^+]$ give us that $\{1\otimes (1+\iota),
\overset{+}{\sigma}\otimes \iota\}$ is a $\Z$-basis for
$\Lambda^{\Gamma^+}$, on which $\iota$ acts via $\begin{pmatrix} 1&g\\
  0&-1\end{pmatrix}$. Identifying the space with $\Z^2$ we can compute
cocycles and coboundaries. Coboundaries are of the form $a_\iota =
(-gb,2b)$ for $b\in \Z$. Cocycles are of the form $a_\iota = (a,b)$
with $2a+gb=0$.  Thus, if $b$ is even then it is a  coboundary, if $b$
is odd 
then $g$ cannot be odd, and so we only get a nontrivial cocycle with $g$
even and $b$ odd. The difference of two nontrivial cocycles has an
even second entry, so it is a coboundary. This proves
$H^1(\Z/2\Z,\Lambda^{\Gamma^+})=  \left\lbrace  \begin{array}{l}
    \{0\}\text{ if }g\text{ is odd}\\ \Z/2\Z\text{ if }g\text{ is
      even}   \end{array}\right. $   as desired.

For the last assertion in the proposition, when $g$ is odd,  by the Schur-Zassenhaus theorem (see \cite[Theorem 7.41]{rotman}) the sequence (\ref{eq:seq_group}) splits and we get our result immediately.
\end{proof}

\subsection{Case $g$ odd}

Now we are ready to show 
\begin{lemma}
If $g$ is odd, then $\displaystyle\tau_T = \frac{1}{1}=1$.
\end{lemma}
\begin{proof}
Using corollary \ref{cor:g_odd} we have that $|H^1(K/\Q,\mathbf{X}^\star(T))|=1$, and $\Sh^1(T)$ has odd order. To show that $\Sh^1(T)$ is trivial, it suffices to show it is $2$-torsion. The cohomology of the sequence \ref{eq:diagT} yields 
\[\xymatrix{H^1(K/\Q,T^\der)\ar[r]& H^1(K/\Q,T)\ar[r]& H^1(K/\Q,\gp_m)=1},\]
where the right equality holds by Hilbert 90. So we have a surjection of $H^1(K/\Q,T^\der)$ onto $H^1(K/\Q,T)$. We claim that $H^1(K/\Q,T)$ is 2-torsion; it suffices so show that $H^1(K/\Q,T^\der)$ is.

Taking the cohomology of the following sequence where the middle term is cohomologically trivial, 
\[\xymatrix{1\ar[r]& \resone_{K/K^+}\gp_m\ar[r]& \res_{K/K^+}\gp_m \ar[r]^-{N_{K/K^+}}& \gp_m\ar[r]& 1},\]
we have $H^1(K/K^+,\resone_{K/K^+}\gp_m)\cong \hat{H}^0(K/K^+,\gp_m)= (K^+)^\times/N_{K/K^+}K^\times.$

This gives us 
\[H^1(K/\Q,T^\der) = H^1(K/\Q,\res_{K/\Q}\resone_{K/K^+}\gp_m) = H^1(K/K^+,\resone_{K/K^+}\gp_m)\cong (K^+)^\times/N_{K/K^+}K^\times.\]

This group is $2$-torsion, hence so is $H^1(K/\Q,T)$ and $\Sh^1(T)$ is a subgroup of the latter. We can conclude that $\Sh(T)$ is a $2$-torsion group of odd order, hence it is trivial. 

We can conclude using \ref{thm:ono_tamagawa_tori} that $\tau_T= \frac{1}{1}=1$.

Note that one need not use corollary \ref{cor:g_odd} to know that
$\Sh^1(T)$ has odd order and hence is trivial. Indeed, given that the
extension $K/K^+$ is quadratic, by the Chebotarev density theorem, we
know that there is a prime $\mathfrak{p}\in K^+$ inert in the
extension $K/K^+$. Since $\mathfrak{p}$ is stable under $\iota$, which
is of order $2$, then its decomposition group $\Gamma(\mathfrak{p})$
has even order, and therefore odd index in $\Gamma$. Now it suffices
to look at the restriction-corestriction sequence
$H^1(\Gamma,T)\rightarrow H^1(\Gamma(\mathfrak{p}),T)\rightarrow
H^1(\Gamma,T)$. The composition of the two maps is just multiplication
by $n=[\Gamma:\Gamma(\mathfrak{p})]$. By definition of $\Sh^1(T)$,
this subgroup of $H^1(\Gamma,T)$ is killed by the restriction map,
hence it is $n$-torsion, and we know $n$ is odd, as desired. 
\end{proof}

\subsection{The case $K/\Q$ cyclic}

\begin{lemma}\label{lem:K_cyclic_tau_1}
When $K$ is cyclic, we have  $\tau(T)=\frac{1}{1}=1$.  In particular, this holds when $K^+/\Q$ is cyclic of odd order.
\end{lemma}
\begin{proof}
 Write $\mathbf{X}^\star(T)= \Z[\Gamma]/L$ as in (\ref{eq:char_description}).

By virtue of $K$ being cyclic, the Tate cohomology is $2$-periodic so 
\[H^1(K/\Q,\mathbf{X}^\star(T) ) \cong H^2(K/\Q, L) \cong \hat{H}^0(K/\Q, L) .\] 
Since $L$ has trivial $\iota$ action, we can see it as the augmentation ideal of $\Z[\Gamma^+]$, which has no $\Gamma^+$-fixed point, as any augmentation ideal. In particular it has no $\Gamma$-fixed point and so $\hat{H}^0(K/\Q, L)=\{0\}$.

Again using the fact that $K$ is cyclic, we get that $\Sh^1(T)$ is
trivial. Indeed, by the Chebotarev density theorem, every cyclic extension has a prime $p\in \Z$ that will stay inert, and therefore  $\Gamma(p)=\Gamma$ where $\Gamma(p)$ is the corresponding decomposition group. Therefore, the map in the definition of $\Sh^1(T)$ is injective.

We can conclude by Theorem \ref{thm:ono_tamagawa_tori} that $ \tau_T = \frac{1}{1}=1$.
\end{proof}

\subsection{Case $g=2$}
\label{seq:g2}
\begin{lemma}
When $g=2$ we have $\tau_T=1$ if $\Gamma$ is cyclic, otherwise $\tau_T=2$.
\end{lemma}
\begin{proof}
The first case is a consequence of Lemma \ref{lem:K_cyclic_tau_1}. If $\Gamma$ isn't cyclic, the only possibility is $\Gamma \cong (\Z/2\Z)^2$.

In that case, proposition \ref{prop:h1_split} gives $H^1(\Gamma,\mathbf{X}^\star(T))\cong \Z/2\Z$.  Concerning the Tate-Shafarevich group, is was shown by Cortella in \cite{cortella_hasse} that  $\Sh^1(T)=\{0\}$ if $g<4$.

Alternatively, since $\Gamma$ is abelian, every proper cyclic subgroup appears as decomposition group. In this specific case, it is a consequence of the Chinese remainder theorem and quadratic reciprocity.  Using SageMath computations (see below), we obtain that the map 
\[\xymatrix{H^2(\Gamma,\mathbf{X}^\star(T))\ar[r]&  H^2(\Z/2\Z\times \{0\},\mathbf{X}^\star(T))\oplus   H^2(\{0\}\times\Z/2\Z,\mathbf{X}^\star(T))}\]
is injective, therefore $\Sh^1(T)=0$.
\end{proof}

\subsection{Computing the Tamagawa numbers with SageMath}
\label{seq:sage}
The second author implemented methods in SageMath to deal with algebraic tori through their character lattices. Those methods should eventually be added to SageMath in a future release.

 Here we briefly describe the computation of the Tamagawa number which arises in Example \ref{ex:g4}, where $K = \Q[T]/f(T)$ for
 \[f(T) = T^3-6T^7+13T^6-10T^5+T^4-30T^3+117T^2-162T+81.\]

 We have $\Gamma = \mathrm{Gal}(K/\Q)= \Z/4\oplus \langle\iota\rangle$ where $\iota$ denotes the complex involution.

Nakayama duality lets us compute the Tamagawa number as a function of the character lattice,
\[\tau_\Q(T) = \frac{|H^1(\Q,\mathbf{X}^\star(T))|}{|\Sh^2(\mathbf{X}^\star(T))|}.\]

Let  $\Lambda$ denote $\mathbf{X}^\star(T)$. We build $\Lambda$ in SageMath by inducing the trivial lattice $\Z=\mathbf{X}^\star(\mathbb{G}_m)$ to $\Gamma$, build the sublattice of zero sum elements of $\iota$-fixed points, and quotient the former by the latter.

We can compute the first cohomology group by computing cocycles as solutions of linear equations in $\Lambda^{|\Gamma|}$. However for this example, Proposition \ref{prop:h1_split} gives us $H^1(\Gamma,\Lambda) = 2$.  For the denominator, we build a method that given $\mathcal{H}$ a collection of subgroups of $\Gamma$, and a $\Gamma$-lattice $L$, computes 

\[\Sh^1_\mathcal{H}(L) = \mathrm{Ker} \left(H^1(\Gamma,L)\rightarrow \bigoplus_{\Delta\in\mathcal{H}} H^1(\Delta,L)\right),\]
by checking what cocycles restrict to coboundaries on all $\Delta\in \mathcal{H}$.

Consider the embedding $\varphi : \Lambda\rightarrow \Z[\Gamma]\otimes_\mathbb{Z} \Lambda$ (with $\Gamma$-action on the left component) via $ a\mapsto \sum_{g\in \Gamma}g\otimes g^{-1}a$. We build \[\Lambda' = \Z[\Gamma]\otimes_\Z \Lambda / \varphi(\Lambda).\]

Since $\Z[\Gamma]\otimes_\Z\Lambda$ is induced, it is cohomologically trivial, hence $\hat{H}^i(\Gamma,\Lambda) = \hat{H}^{i-1}(\Gamma,\Lambda)$ for all $i\in\Z$. Consequently we have 
\[ \Sh_\mathcal{H}^1(\Lambda') \supset \Sh^1(\Lambda')=\Sh^2(\Lambda). \]

We take $\mathcal{H}$ to be the list of cyclic subgroups of $\Gamma$. They all arise as decomposition groups since $\Gamma$ is cyclic. Using the program, we get $|\Sh_\mathcal{H}^1(\Lambda')|=1\geq \Sh^2(\Lambda)\geq 1$, hence 
\[\tau(T) = \frac{|H^1(\Gamma,\Lambda)|}{|\Sh^2(\Lambda)|}= \frac{2}{1}=2.\]

\subsection{The numerator in Ono's formula: general case}
\label{seq:general}
We give some indications for the general case in which  $T$ is
described by a CM algebra $K = \bigoplus_{i=1}^t K_i$, each $K_i$
being a CM field. The Rosati involution on $K$ is still denoted as
$\iota$, with fixed subalgebra $K^+ = \bigoplus_{i=1}^t K_i^+$. In
this section we denote by $\Gamma$ the absolute Galois group of
$\mathbb{Q}$. Our modest aim is to understand $\left| H^1(\Gamma,
  \mathbf{X}^\star(T)) \right|$ through Kottwitz's isomorphism (see
\cite{kottwitz:84a} (2.4.1) and \S 2.4.3): 
\[ H^1(\Gamma, \mathbf{X}^\star(T)) \iso \pi_0(\hat{T}^\Gamma), \]
where $\hat{T}$ is the dual $\mathbb{C}$-torus. This isomorphism is valid for all tori.

To describe $\mathbf{X}^\star(T)$, we first write $T$ as
\[ T = (\mathbb{G}_m \times T^{\mathrm{der}}) \big/ \left\{ (z,z): z \in \mu_2 \right\}, \quad \mu_2 := \{\pm 1\}. \]

Choose a subset $\Phi = \bigsqcup_{i=1}^t \Phi_i$ of $\mathrm{Hom}_{\mathbb{Q}\text{-alg}}(K, \overline{\mathbb{Q}})$, such that $\Phi_i \subset \mathrm{Hom}_{\mathbb{Q}\text{-alg}}(K_i, \overline{\mathbb{Q}})$ and
\[ \mathrm{Hom}_{\mathbb{Q}\text{-alg}}(K_i, \overline{\mathbb{Q}}) = \Phi_i \sqcup \Phi_i \iota \]
for all $i = 1, \ldots, t$. Note that $|\Phi| = g$. It is well-known that
\[ \mathbf{X}^\star(T^{\mathrm{der}}) = \bigoplus_{\phi \in \Phi} \mathbb{Z} \epsilon_\phi \]
for some basis $\{\epsilon_\phi \}_{\phi \in \Phi}$. Then $\Gamma$ permutes $\{\pm \epsilon_\phi \}_{\phi \in \Phi}$ by
\begin{equation}\label{eqn:Galois-Phi-action}
  \sigma \epsilon_\phi = \begin{cases}
    \epsilon_\psi, & \text{if } \sigma\phi = \psi \in \Phi \\
    -\epsilon_\psi, & \text{if } \sigma\phi = \psi\iota \in \Phi\iota,
  \end{cases}
  \quad \phi, \psi \in \Phi.
\end{equation}
The inclusion $\mu_2 \hookrightarrow T^{\mathrm{der}}$ corresponds to the map
\begin{align*}
  \mathbf{X}^\star(T^{\mathrm{der}}) & \twoheadrightarrow \mathbf{X}^\star(\mu_2) = \mathbb{Z}/2\mathbb{Z} \\
  \sum_{\phi \in \Phi} x_\phi \epsilon_\phi & \mapsto \sum_{\phi \in \Phi} x_\phi \quad \bmod \;2.
\end{align*}

Write $\mathbf{X}^\star(\mathbb{G}_m) = \mathbb{Z}\eta$, where $\eta$ is the standard generator. Applying Cartier duality to the exact sequence
\[ 1 \to \mu_2 \to \mathbb{G}_m \times T^{\mathrm{der}} \to T \to 1, \]
we obtain
\begin{align*}
  \mathbf{X}^\star(T) & = \left\{ t\eta + \sum_{\phi \in \Phi} x_\phi \epsilon_\phi : t + \sum_\phi x_\phi \; \in 2\mathbb{Z} \right\} \subset \mathbf{X}^\star(\mathbb{G}_m) \oplus \mathbf{X}^\star(T^{\mathrm{der}}), \\
  \text{a basis of } \mathbf{X}^\star(T): & \quad \{ 2\eta \} \sqcup \left\{ \eta + \epsilon_\phi : \phi \in \Phi \right\}.
\end{align*}
The element $2\eta$ is surely $\Gamma$-invariant. On the other hand, for $\phi, \psi \in \Phi$ and $\sigma \in \Gamma$, we derive from \eqref{eqn:Galois-Phi-action} that
\begin{equation}\label{eqn:Galois-X-action}
  \sigma (\eta + \epsilon_\phi) = \begin{cases}
    \eta + \epsilon_\phi, & \text{if } \sigma\phi = \psi \in \Phi \\
    2\eta -(\eta + \epsilon_\psi), & \text{if } \sigma\phi = \psi\iota \in \Phi\iota.
  \end{cases}
\end{equation}
The $\Gamma$-action on $\hat{T} := \mathbf{X}^\star(T) \otimes \mathbb{C}^\times \iso \mathbb{C}^\times \times (\mathbb{C}^\times)^\Phi$ is thus
\[ \sigma \cdot (z, \overbrace{1, \ldots, \underbrace{w}_{\phi}, \ldots, 1}^{\Phi}) =
\begin{cases}
  (z, 1, \ldots, \underbrace{w}_{\psi}, \ldots, 1 ), & \text{if } \sigma\phi = \psi \in \Phi \\
  (zw, 1, \ldots, \underbrace{w^{-1}}_{\psi}, \ldots, 1 ), & \text{if } \sigma\phi = \psi\iota \in \Phi\iota.
\end{cases}\]

Recall from the description of $\mathbf{X}^\star(T^{\mathrm{der}})$ that $\widehat{T^{\mathrm{der}}}$ can be identified with $(\mathbb{C}^\times)^\Phi$. Note that $(\widehat{T^{\mathrm{der}}})^\Gamma$ is finite since $T^{\mathrm{der}}$ is anisotropic.

\begin{lemma}\label{prop:T-mu2}
  There is a canonical isomorphism $(\widehat{T^{\mathrm{der}}})^\Gamma \iso \mu_2^t$ characterized as follows. For $(a_i)_{i=1}^t \in \mu_2^t$, the corresponding $\hat{t} = (\hat{t}_\phi)_{\phi \in \Phi} \in (\widehat{T^{\mathrm{der}}})^\Gamma$ is specified by
  \[ \forall 1 \leq i \leq t, \quad \phi \in \Phi_i \implies \hat{t}_\phi = a_i. \]
\end{lemma}
\begin{proof}
  Shapiro's lemma reduces the computation of $(\widehat{T^{\mathrm{der}}})^\Gamma$ or $H^1(\Gamma, \mathbf{X}^\star(T))$ to the easy case $t=1$ over the base field $K^+$.
\end{proof}

By dualizing $1 \to T^{\mathrm{der}} \to T \to \mathbb{G}_m \to 1$ into $1 \to \mathbb{C}^\times \to \hat{T} \to \widehat{T_0} \to 1$, then taking $\Gamma$-invariants, we obtain the exact sequence
\[ 1 \to \mathbb{C}^\times \to \hat{T}^\Gamma \to (\widehat{T^{\mathrm{der}}})^\Gamma. \]
It induces
\[ \pi_0(\hat{T}^\Gamma) = \hat{T}^\Gamma /\mathbb{C}^\times \hookrightarrow (\widehat{T^{\mathrm{der}}})^\Gamma = \pi_0((\widehat{T^{\mathrm{der}}})^\Gamma). \]

For each $\sigma \in \Gamma$, set $\Phi(\sigma) := \left\{ \phi \in \Phi: \sigma\phi \notin \Phi \right\}$. For each $i$, set $\Phi_i(\sigma) := \Phi(\sigma) \cap \Phi_i$.

\begin{proposition}\label{prop:numerator-Ono}
  For any $(a_i)_i \in \mu_2^t$, the corresponding element $\hat{t} \in (\widehat{T^{\mathrm{der}}})^\Gamma$ belongs to the image of $\hat{T}^\Gamma / \mathbb{C}^\times$ if and only if
  \[ A(a_1, \ldots, a_t; \sigma) := \sum_{\substack{1 \leq i \leq t \\ a_i = -1}} |\Phi_i(\sigma)| \; \in 2\mathbb{Z} \]
  for all $\sigma \in \Gamma_F$.
\end{proposition}
\begin{proof}
  Identify $\widehat{T^{\mathrm{der}}}$ with $(\mathbb{C}^\times)^\Phi$. Identify $\hat{T}$ with $\mathbb{C}^\times \times (\mathbb{C}^\times)^\Phi$ using the basis $\{2\eta \} \sqcup \{ \eta + \epsilon_\phi : \phi \in \Phi \}$ of $\mathbf{X}^\star(T)$; the homomorphism $\hat{T} \to \widehat{T^{\mathrm{der}}}$ is simply the projection.

  Note that $\hat{t}$ is the image of $(1, \hat{t}) \in \hat{T}$. It comes from $\hat{T}^\Gamma / \mathbb{C}^\times$ if and only if $(1, \hat{t})$ (or any other preimage) is $\Gamma$-invariant. For all $\sigma \in \Gamma$, Lemma \ref{prop:T-mu2} and the description \eqref{eqn:Galois-X-action} lead to
  \[ \sigma \cdot (1, \hat{t}) = \left( (-1)^{A(a_1, \ldots, a_t; \sigma)}, \hat{t} \right). \]
  The assertion follows at once.
\end{proof}

To illustrate the use of proposition \ref{prop:numerator-Ono}, we prove the following
\begin{proposition}
  If $t=1$ and $g$ is odd, then $H^1(\Gamma, \mathbf{X}^\star(T)) \iso \pi_0(\widehat{T}^\Gamma)$ is trivial.
\end{proposition}
\begin{proof}
  It suffices to show $A(-1, \ldots, -1; c) = |\Phi(c)| \notin 2\mathbb{Z}$ where $c \in \Gamma$ is the complex conjugation. Indeed, $c\phi = \phi\iota$ for all $\phi \in \Phi$ by generalities on CM fields, hence $\Phi(c) = \Phi$ has $g$ elements, which is odd.
\end{proof}

Note that $K$ is not assumed to be Galois over $\mathbb{Q}$.

Kottwitz's theory also relates Tate--Shafarevich groups to similar objects attached to dual tori; see \S 4 of \cite{kottwitz:84a}. Nevertheless, we are not yet able to determine the Tate--Shafarevich group of $T$ by this approach in the non-Galois case.

\bibliographystyle{hamsalpha}
\bibliography{latest_biblio1}

\def\polhk#1{\setbox0=\hbox{#1}{\ooalign{\hidewidth
  \lower1.5ex\hbox{`}\hidewidth\crcr\unhbox0}}}
\providecommand{\bysame}{\leavevmode\hbox to3em{\hrulefill}\thinspace}
\providecommand{\MR}{\relax\ifhmode\unskip\space\fi MR }
\providecommand{\MRhref}[2]{%
  \href{http://www.ams.org/mathscinet-getitem?mr=#1}{#2}
}
\providecommand{\href}[2]{#2}
\begin{thebibliography}{Kot84b}

\bibitem[AG17]{achtergordon17}
Jeffrey~D. Achter and Julia Gordon, \emph{Elliptic curves, random matrices and
  orbital integrals}, Pacific J. Math. \textbf{286} (2017), no.~1, 1--24, With
  an appendix by S. Ali Altu\u g. \MR{3582398}

\bibitem[AW13]{achterwong}
Jeffrey~D. Achter and Siman Wong, \emph{Quotients of elliptic curves over
  finite fields}, Int. J. Number Theory \textbf{9} (2013), no.~6, 1395--1412.
  \MR{3103894}

\bibitem[AW15]{achterwilliams15}
Jeffrey~D. Achter and Cassandra Williams, \emph{Local heuristics and an exact
  formula for abelian surfaces over finite fields}, Canad. Math. Bull.
  \textbf{58} (2015), no.~4, 673--691. \MR{3415659}

\bibitem[Bit11]{bitan}
Rony~A. Bitan, \emph{The discriminant of an algebraic torus}, J. Number Theory
  \textbf{131} (2011), no.~9, 1657--1671. \MR{2802140 (2012g:11118)}

\bibitem[Bou85]{Bourbaki:commalg}
N.~Bourbaki, \emph{El\'em\'ents de math\'ematique. alg\`ebre commutative},
  Masson, 1985.

\bibitem[Clo90]{clozel:fl}
Laurent Clozel, \emph{The fundamental lemma for stable base change}, Duke Math.
  J. \textbf{61} (1990), no.~1, 255--302.

\bibitem[Con14]{conradsga3}
Brian Conrad, \emph{Reductive group schemes}, Autour des sch\'emas en groupes.
  {V}ol. {I}, Panor. Synth\`eses, vol. 42/43, Soc. Math. France, Paris, 2014,
  pp.~93--444. \MR{3362641}

\bibitem[Cor97]{cortella_hasse}
Anne Cortella, \emph{The {H}asse principle for the similarities of bilinear
  forms}, Algebra i Analiz \textbf{9} (1997), no.~4, 98--118. \MR{1604012}

\bibitem[Deu41]{deuring41}
Max Deuring, \emph{Die {T}ypen der {M}ultiplikatorenringe elliptischer
  {F}unktionenk\"{o}rper}, Abh. Math. Sem. Hansischen Univ. \textbf{14} (1941),
  197--272. \MR{5125}

\bibitem[FLN10]{langlands-frenkel-ngo}
Edward Frenkel, Robert Langlands, and B{\'a}o~Ch{\^a}u Ng{\^o}, \emph{Formule
  des traces et fonctorialit\'e: le d\'ebut d'un programme}, Ann. Sci. Math.
  Qu\'ebec \textbf{34} (2010), no.~2, 199--243. \MR{2779866 (2012c:11240)}

\bibitem[Gek03]{gekeler03}
Ernst-Ulrich Gekeler, \emph{Frobenius distributions of elliptic curves over
  finite prime fields}, Int. Math. Res. Not. (2003), no.~37, 1999--2018.

\bibitem[GG99]{gan-gross:haar}
Wee~Teck Gan and Benedict~H. Gross, \emph{Haar measure and the {A}rtin
  conductor}, Transactions of the AMS \textbf{351} (1999), no.~4, 1691--1704.

\bibitem[Gor20]{jg:singapore}
Julia Gordon, \emph{Orbital integrals and normalization of measures}, to appear
  in proceedings of "On the Langlands Program: Endoscopy and Beyond", NUS.

\bibitem[Gro97]{gross:motive}
Benedict~H. Gross, \emph{On the motive of a reductive group}, Invent. Math.
  \textbf{130} (1997), {287--313}.

\bibitem[Gro98]{gross:satake}
Benedict~H. Gross, \emph{On the {S}atake isomorphism}, Galois representations
  in arithmetic algebraic geometry ({D}urham, 1996), London Math. Soc. Lecture
  Note Ser., vol. 254, Cambridge Univ. Press, Cambridge, 1998, pp.~223--237.
  \MR{1696481}

\bibitem[GS06]{gille_csa}
Philippe Gille and Tam\'{a}s Szamuely, \emph{Central simple algebras and
  {G}alois cohomology}, Cambridge Studies in Advanced Mathematics, vol. 101,
  Cambridge University Press, Cambridge, 2006. \MR{2266528}

\bibitem[GSY20]{gsy20}
Jia-Wei Guo, Nai-Heng Sheu, and Chia-Fu Yu, \emph{Class numbers of {CM}
  algebraic tori, {CM} abelian varieties and components of unitary {S}himura
  varieties}, Nagoya Mathematical Journal (2020), 126.

\bibitem[GW19]{gerhardwilliams19}
Jonathan Gerhard and Cassandra Williams, \emph{Local heuristics and an exact
  formula for abelian varieties of odd prime dimension over finite fields}, New
  York J. Math. \textbf{25} (2019), 123--144. \MR{3904880}

\bibitem[GY00]{gan-yu}
Wee~Tack Gan and Jiu-Kang Yu, \emph{Group schemes and local densities}, Duke
  Math. Journal \textbf{105} (2000), no.~3, 497--524.

\bibitem[How20]{howevariation}
Everett~W. Howe, \emph{Variations in the distribution of principally polarized
  abelian varieties among isogeny classes}, 2020, \mbox{arXiv:2005.14365}.

\bibitem[Kis17]{kisin17}
Mark Kisin, \emph{{${\rm mod}\,p$} points on {S}himura varieties of abelian
  type}, J. Amer. Math. Soc. \textbf{30} (2017), no.~3, 819--914. \MR{3630089}

\bibitem[Kot82]{kottwitz82}
Robert~E. Kottwitz, \emph{Rational conjugacy classes in reductive groups}, Duke
  Math. J. \textbf{49} (1982), no.~4, 785--806. \MR{683003 (84k:20020)}

\bibitem[Kot84a]{kottwitz:84a}
Robert Kottwitz, \emph{Stable trace formula: cuspidal tempered terms}, Duke
  Math Journal \textbf{51} (1984), no.~3, 611--650.

\bibitem[Kot84b]{kottwitz:shimura}
Robert~E. Kottwitz, \emph{Shimura varieties and twisted orbital integrals},
  Math. Ann. \textbf{269} (1984), no.~3, 287--300. \MR{761308}

\bibitem[Kot90]{kottwitz90}
\bysame, \emph{Shimura varieties and {$\lambda$}-adic representations},
  Automorphic forms, {S}himura varieties, and {$L$}-functions, {V}ol.\ {I}
  ({A}nn {A}rbor, {MI}, 1988), Perspect. Math., vol.~10, Academic Press,
  Boston, MA, 1990, pp.~161--209. \MR{1044820}

\bibitem[Kot92]{kottwitz92}
\bysame, \emph{Points on some {S}himura varieties over finite fields}, J. Amer.
  Math. Soc. \textbf{5} (1992), no.~2, 373--444. \MR{1124982 (93a:11053)}

\bibitem[Kot05]{kottwitz:clay}
\bysame, \emph{Harmonic analysis on reductive {$p$}-adic groups and {L}ie
  algebras}, Harmonic analysis, the trace formula, and {S}himura varieties,
  Clay Math. Proc., vol.~4, Amer. Math. Soc., Providence, RI, 2005,
  pp.~393--522. \MR{2192014 (2006m:22016)}

\bibitem[Lan73]{langlands:antwerp}
R.~P. Langlands, \emph{Modular forms and {$\ell $}-adic representations},
  361--500. Lecture Notes in Math., Vol. 349. \MR{0354617}

\bibitem[LYY21]{lyy21}
Pei-Xin Liang, Hsin-Yi Yang, and Chia-Fu Yu, \emph{On tamagawa numbers of cm
  tori}, 2021, \mbox{arXiv:2109.04121}.

\bibitem[Mar21]{marseglia21}
Stefano Marseglia, \emph{Computing square-free polarized abelian varieties over
  finite fields}, Math. Comp. \textbf{90} (2021), no.~328, 953--971.
  \MR{4194169}

\bibitem[Oes82]{oesterle}
Joseph Oesterl\'{e}, \emph{R\'{e}duction modulo {$p^{n}$} des sous-ensembles
  analytiques ferm\'{e}s de {${\bf Z}^{N}_{p}$}}, Invent. Math. \textbf{66}
  (1982), no.~2, 325--341. \MR{656627}

\bibitem[Ono61]{ono:arithmetic_tori}
Takashi Ono, \emph{Arithmetic of algebraic tori}, Ann. of Math. (2) \textbf{74}
  (1961), no.~1, 101--139.

\bibitem[Ono63]{ono_tamagawa_tori}
\bysame, \emph{On the {T}amagawa number of algebraic tori}, Ann. of Math. (2)
  \textbf{78} (1963), 47--73. \MR{0156851}

\bibitem[PR91]{platonov-rapinchuk:AlgGroupsAndNT}
Vladimir Platonov and Andrei Rapinchuk, \emph{Algebraic groups and number
  theory}, Academic Press, Inc., San Diego, 1991.

\bibitem[Rot95]{rotman}
Joseph~J. Rotman, \emph{An introduction to the theory of groups}, fourth ed.,
  Graduate Texts in Mathematics, vol. 148, Springer-Verlag, New York, 1995.
  \MR{1307623}

\bibitem[Rü20]{rud20}
Thomas Rüd, \emph{Explicit tamagawa numbers for certain algebraic tori over
  number fields}, 2020, \mbox{arXiv:2009.04431}.

\bibitem[Ser81]{serre:chebotarev}
Jean-Pierre Serre, \emph{Quelques applications du th\'eor\`eme de densit\'e de
  {C}hebotarev}, Inst. Hautes \'Etudes Sci. Publ. Math. (1981), no.~54,
  323--401.

\bibitem[Tat66]{tateendff}
John Tate, \emph{Endomorphisms of abelian varieties over finite fields},
  Invent. Math. \textbf{2} (1966), 134--144. \MR{MR0206004 (34 \#5829)}

\bibitem[Vey92]{veys:measure}
Willem Veys, \emph{Reduction modulo {$p^n$} of {$p$}-adic subanalytic sets},
  Math. Proc. Cambridge Philos. Soc. \textbf{112} (1992), no.~3, 483--486.
  \MR{1177996 (93i:11142)}

\bibitem[Vos98]{voskresenskii}
V.~E. Voskresenski\u{\i}, \emph{Algebraic groups and their birational
  invariants}, Translations of Mathematical Monographs, vol. 179, American
  Mathematical Society, Providence, RI, 1998, Translated from the Russian
  manuscript by Boris Kunyavski [Boris \`E. Kunyavski\u{\i}]. \MR{1634406}

\bibitem[Wat69]{waterhouse:69}
William~C. Waterhouse, \emph{Abelian varieties over finite fields}, Ann. Sci.
  \'Ecole Norm. Sup. (4) \textbf{2} (1969), 521--560. \MR{0265369}

\bibitem[Wei82]{weil:adeles}
Andr{\'e} Weil, \emph{Adeles and algebraic groups}, Progress in mathematics,
  vol.~23, Birkh\"auser, 1982.

\bibitem[XY21]{xueyu21}
Jiang~Wei Xue and Chia~Fu Yu, \emph{On counting certain abelian varieties over
  finite fields}, Acta Math. Sin. (Engl. Ser.) \textbf{37} (2021), no.~1,
  205--228. \MR{4204542}

\bibitem[Yu12]{yu12}
Chia-Fu Yu, \emph{Superspecial abelian varieties over finite prime fields}, J.
  Pure Appl. Algebra \textbf{216} (2012), no.~6, 1418--1427. \MR{2890511}

\end{thebibliography}

\end{document}